\documentclass[11pt]{article}
\usepackage[T1]{fontenc}
\usepackage[utf8]{inputenc}
\usepackage{amssymb,graphicx,mathtools,amsfonts}
\usepackage{amsmath}
\usepackage{amsthm}
\usepackage{xcolor}
\usepackage[colorlinks]{hyperref}
\hypersetup{
	colorlinks=true,
	linkcolor=blue,      
	citecolor=red
	}
\usepackage{dsfont}
\usepackage{a4wide}
\usepackage{esint}
\catcode`\@=11
\def\theequation{\@arabic{\c@section}.\@arabic{\c@equation}}
\catcode`\@=12
\graphicspath{{images/}}
\newtheorem{theorem}{Theorem}[section] 
\newtheorem{lemma}{Lemma}[section] 
\newtheorem{definition}{Definition}[section] 

\newtheorem{remark}{Remark}[section]

\allowdisplaybreaks

\newcommand{\R}{\mathbb{R}}
\newcommand{\N}{\mathbb{N}}
\newcommand{\X}{\mathcal{X}_{0}}
\newcommand{\Q}{\textit{Q}}
\newcommand{\A}{\mathcal{A}_{k_{n+1}}}
\newcommand{\Nh}{\mathcal{N}_\la}
\newcommand{\rh}{\rho_{\X}}
\newcommand{\lmt}{\lim_{n\to\infty}}

\newcommand{\sm} {\setminus}

\newcommand{\ve}{\varepsilon}
\newcommand{\pa} {\partial}
\newcommand{\al} {\alpha}

\newcommand{\Th} {\Theta}

\newcommand{\ga} {\gamma}

\newcommand{\om} {\Omega}
\newcommand{\De} {\Delta}
\newcommand{\la} {\lambda}
\newcommand{\La} {\Lambda}
\newcommand{\noi} {\noindent}
\newcommand{\na} {\nabla}

\newcommand{\oline} {\overline}
\newcommand{\ds} {\displaystyle}
\newcommand{\ra} {\rightarrow}
\newcommand{\real}{\mathbb{R}}

\newcommand{\rnn}{\mathbb{R}^{N}}
\newcommand{\lv}{\lVert}
\newcommand{\rv}{\rVert}
\newcommand{\weak}{\rightharpoonup}

\newcommand{\lb}{\left(}
\newcommand{\rb}{\right)}
\newcommand{\lsb}{\left[}
\newcommand{\rsb}{\right]}
\newcommand{\grad}{\nabla}

\title{Multiplicity results for mixed local-nonlocal variable exponent problem involving singular and superlinear term}

\author{Shammi Malhotra\thanks{Department of Mathematics, Indian Institute of Technology Delhi, Hauz Khas New Delhi 110016,  India, \texttt{maz228085@maths.iitd.ac.in}}, Ambesh Kumar Pandey\thanks{Department of Mathematics, Indian Institute of Technology Delhi, Hauz Khas New Delhi 110016,  India, \texttt{pandey.ambesh190@gmail.com}} and K. Sreenadh\thanks{Department of Mathematics, Indian Institute of Technology Delhi, Hauz Khas New Delhi 110016,  India, \texttt{sreenadh@maths.iitd.ac.in}}}
\date{}

\begin{document}

\date{}
\maketitle

\allowdisplaybreaks

 \begin{abstract}
\noi In this paper, we study a class of quasilinear elliptic equations involving both local and nonlocal operators with variable exponents. The problem exhibits singular nonlinearities along with a subcritical superlinear growth term and a parameter $\la$. We study the existence of multiple solutions with the help of variational methods by restricting the associated energy functional on appropriate subsets of the Nehari manifold. Using the topological index and the structure of the fibering maps, we analyse a key splitting property of the associated Nehari manifold. This decomposition allows us to establish the existence of two distinct solutions. Additionally, we establish the $L^\infty$-bound for the solutions by adopting the De Giorgi iteration method.

\medskip
\noi\textbf{Keywords:} Mixed local nonlocal operator; Variable exponents; Nehari manifold; Fibering map; Singularity; Degree theory; Regularity.

\medskip
\noi\textbf{Mathematics Subject Classification:} 35J60, 35J50, 35J75, 55M25, 35R11, 35J92
\end{abstract}

\section{Introduction}
In this article, we consider the following mixed local 
and nonlocal elliptic problem involving 
variable exponents and a singularity
	\begin{equation}\label{eq1}\tag{$\mathcal{Q}_{\lambda}$}
   \left\{
	\begin{aligned}
		-\De_{p(x)}u+\lb -\De\rb^s_{p(x)}u 
&=\lambda a(x)u^{-\gamma(x)}+b(x)u^{r(x)-1},\,
u>0 \ \text{ in } \om,\\
u &= 0 \qquad \text{ in } \rnn\backslash\om,
	\end{aligned}
    \right.
\end{equation}
where
$\om \subset \rnn$
is a bounded domain with Lipschitz boundary,
$\la>0, s\in (0,1),$
the exponent $p(x):=p(x,x)$
for all $x\in\rnn$
where $p(x,y)\in C(\rnn\times\rnn)$
satisfying  the following assumptions
\begin{enumerate}
    
\item[$(P_1)$] Symmetry: $p(x,y)=p(y,x)$
for all $x,y\in \rnn$;
    
\item[$(P_2)$] $1<\inf_{(x,y)\in 
\rnn\times\rnn}p(x,y)\le \sup_{(x,y)\in 
\rnn\times\rnn}p(x,y)<N$.
\end{enumerate}
The weight function $a(x)$
is positive, and $b(x)$
is nonnegative. The variable exponents
$\ga(x)$
and
$r(x)\in C(\oline\om,\real),$
satisfying some assumptions stated later. 
The local operator is defined as 
$$
\De_{p(x)}u:=\text{div}(|\na u|^{p(x)-2}\na u),
$$
while the nonlocal part, the fractional
$p(x)$-Laplacian is defined as the Cauchy 
principle value, which is given by
$$
\lb-\De\rb^s_{p(x)}u:=\text{P.V.}
\int_{\mathbb{R}^{N}}
\frac{|u(x)-u(y)|^{p(x,y)-2}(u(x)-u(y))}
{|x-y|^{N+sp(x,y)}}\,dy,\quad x\in\rnn.
$$

 Lately, the combination of local and nonlocal operators 
 has emerged as a crucial framework, as they model 
 many physical phenomena, such as describing diverse 
 processes like the diffusion of biological populations.
 Presence of two types of dispersal diffuses the population 
 into a two-fold manner, namely a classical one 
 (i.e., Brownian motion) and a long-range 
 (i.e., L\'{e}vy flight) which can be effectively 
 modelled by Laplacian and fractional Laplacian 
 respectively (see
\cite{ valdinoci_logistic_equation, 
 valdinoci_ecological_niche}).
 It also helps in the study of anisotropic heat 
 transport in reversed shear magnetic fields 
 \cite{blazevski_anisotropic_heat_transport}.

Operators involving variable growth are extensively studied due 
to the precision in modelling various phenomena where the 
property of the subject under consideration 
depends on the point of observation.
For instance,
In image restoration 
(see \cite{Image_restore}),
let
$I$
be an image input representing the shades of gray over a domain
$\om \subset \real^2$
which is made up of the true image corrupted by noise.
Also, assume that noise is additive, i.e.,
$I = u +$noise, where
$u$
is the true image.
To reduce the effect of noise, one can apply a 
smoothing process by minimizing the following 
energy functional
\begin{equation*}
    \mathcal{E}(u) = 
\int_{\om} ( |\grad u(x)|^{p(x)} + |u(x) - I(x)|^2 )dx,
 \text{ with } 1 \leq p(x) \leq 2.
\end{equation*}
In regions likely to have edges, the values of
$p(x)$
are closer to
$1,$
while in other areas, they tend to be closer to
$2$.
Another application arises in the study of 
Electrorheological fluids
\cite{ruzicka2007electrorheological}, 
a special type of viscous fluid whose 
mechanical properties can change dramatically 
depending on an applied electric field.

Over the past few decades, problems involving 
singularity have garnered much research interest.
Such equations first appeared in
\cite{fulks}, where the authors proved an 
existence result for a singular equation 
modelling the steady-state temperature 
distribution in an electrically conducting medium.
Singular elliptic equations also arise in 
applications such as thermo-conductivity
\cite{fulks} and signal transmissions
\cite{signal}.
One of the most notable applications of 
singular equations is modelling boundary 
layers in the theory of non-Newtonian fluids, 
especially in the case of pseudo-plastic fluids
\cite{pseudo}.
For additional applications, 
we refer the reader to
\cite{Singular_book,sing_review} 
and the references therein.

 The study of purely singular problems 
 governed by the equation
 	\begin{equation*}
   \left\{
	\begin{aligned}
		-\De u&=f(x)u^{-\gamma},\, u>0 
		\quad  \text{ in } \om,\\
		u &= 0 \qquad \text{ in } \pa\om,
	\end{aligned}
    \right.
\end{equation*}
 has a long history dating to works such as
\cite{fulks,stuart}.
These types of equations gained a lot of 
attention after the pioneering work of Crandall 
\textit{et al.}
\cite{crandaal}.
In
\cite{mckenna}, Lazer and McKenna proved that 
the unique solution belongs to
$W_0^{1,2}(\om)$
if and only if
$0<\ga<3,$
and that the solution lies in
$C^1(\oline\om)$
when
$0<\ga<1.$
Since then, substantial progress in various 
directions involving singularities has occurred.
For developments in the local case, we refer to
\cite{boccado2010, coclite, lair2, zhang2004}, 
and for the nonlocal case, see
\cite{Arora_nonlocal, CANINO_2017,Giacomoni_2019_nonlocal} 
and references within.

 We now turn our attention towards singular local 
 and nonlocal equations involving subcritical or 
 critical perturbation, whose prototype is 
given by an equation of the form
 \begin{equation}\label{eq_local_sing_lower}
	\begin{aligned}
		-\alpha\De_{p}u+\beta\lb -\De\rb^s_{p}u 
&=\lambda f(x) u^{-\gamma}+ \mu u^{r}
	\end{aligned}
\end{equation}
where
$\alpha,\beta\ge0;\la,\mu,\ga>0$
and
$f$
is given datum.
In the local case
$(\beta=0),$
problem \eqref{eq_local_sing_lower} 
was studied by Yijing \textit{et al.}
\cite{sun2001} for
$\al=1,p=2, 0<\ga<1,f\in L^2(\om)$
and
$1<r<2^*-1.$
The authors proved the existence of at least 
two positive solutions with the help of 
fibering map analysis and the Nehari manifold.
In
\cite{Haito,hirano_2004}, 
authors extended the results of
\cite{sun2001} by considering the cases
$\ga=1$
and
$r\le2^*-1.$
While Hirano \textit{et al.}
\cite{hirano_2004} used a similar 
approach involving the study of a fibering 
map on a subset of the Nehari manifold, Haito
\cite{Haito} took a different strategy by 
combining sub-supersolution and variational 
techniques to obtain the existence of at 
least two positive solutions.
For further study, we refer to
\cite{arcoya2014,Bal_singular_pertub,Jacques_local_pertub, hirano2008}.
In the nonlocal case
$(\alpha=0),$
Barrios \textit{et al.}
\cite{Barrios_2015} studied problem 
\eqref{eq_local_sing_lower} for 
$\beta=1,p=2, \ga>0$
and
$r>1.$
They proved the existence of a 
solution for all values of
$r$
and established the existence of 
multiple solutions when
$$
r<2_s^*-1=
\frac{Np}{N-sp}.
$$
The multiplicity results for
$r=2_s^*-1,$
which corresponds to the critical exponent, 
were later addressed in
\cite{Mukherjee_2016} for
$0<\ga\le 1,$
in
\cite{Mukherjee_2019} for
$0<\ga<3$
and in
\cite{Mukherjee_2017} for all
$\ga>0.$
The quasilinear case involving the fractional
$p$-Laplacian case was studied in
\cite{Mukherjee_nlocal_p}.
In addition, Goyal
\cite{Goyal_1} established the existence of 
multiple solutions using fibering map 
analysis and the Nehari manifold technique.

Problems involving singular nonlinearity are 
relatively less explored in the variable 
exponent setting.
In
\cite{C1_alpha_reg, Saoudi_2017}, the authors 
studied the existence of multiple solutions 
for the equation
\begin{equation*}
   \left\{
	\begin{aligned}
-\De_{p(x)}u&=\lambda a(x)u^{-\gamma(x)}+b(x)
u^{r(x)-1},\,u>0 \qquad \text{ in } \om,\\
u &= 0 \qquad \text{ in } \rnn\backslash\om,
	\end{aligned}
    \right.
\end{equation*}
for
$a(x)=b(x)=1$
and
$p(x)<r(x)<p^*(x).$
Adapting an approach similar to that of
\cite{Haito}, they obtained the existence 
of multiple solutions.
Additionally, they established
$C^{1,\alpha}(\oline\om)$
regularity results for the positive weak solutions.
For the nonlocal case involving the fractional
$p(x)$-Laplace operator, we refer to
\cite{Chammem_2}, where the authors used Nehari 
manifold and variational techniques to obtain 
multiplicity results.
These works were motivated by the earlier study in
\cite{Sweta_2020}.
For a broader perspective on results concerning 
variable exponent spaces, we refer interested 
readers to the comprehensive review by R\u{a}dulescu
\cite{RADULESCU_2015}.

Before presenting our results, 
we briefly review the existing literature 
related to the constant exponent case 
corresponding to problem \eqref{eq1}, which is given as
\begin{equation*}\label{eq_constant}\tag{$\mathcal{R}_{\lambda}$}
   \left\{
	\begin{aligned}
-\De_{p}u+\lb -\De\rb^s_{p}u &=\lambda f(x) 
u^{-\gamma}+ \beta u^{r-1},\,u>0 \qquad \text{ in } \om,\\
		u &= 0 \qquad \text{ in } \rnn\backslash\om.
	\end{aligned}
    \right.
\end{equation*}
 There exists a vast literature concerning 
 the problem related to \eqref{eq_constant}.
For
$p=2,$
in
\cite{arora2025combined}, 
Arora and R\v{a}dulescu studied the 
unperturbed singular problem, i.e.,
$\beta = 0$
and obtained existence and nonexistence results 
together with power and exponential type Sobolev 
regularity results.
Moreover, they analyzed the boundary behaviour 
of the weak solution depending on the 
regularity of source data
$f$
and the exponent
$\gamma.$
Subsequently, Huang and Hajaiej
\cite{Shuibo} generalized these results by 
considering a general, possibly singular nonlinearity
$h(u).$
Later, for the mixed local and nonlocal
$p$-Laplace operator, Garain and Ukhlov
\cite{Garain1} proved the existence, uniqueness, 
regularity and symmetry of weak solutions.
They also established a mixed Sobolev inequality 
involving a singular term, which provides a 
necessary and sufficient condition for the 
existence of the solutions.
An interesting consequence of this study is 
the equivalence between the existence of 
solutions of singular mixed local and nonlocal
$p$-Laplace equations and singular
$p$-Laplace equations.
For
$\beta = 1$
and
$p =2,$
Garain
\cite{Garain2} considered the subcritical
 perturbation case, 
i.e.,
$r  \in (2, 2^*)$
where
$2^* = 
\frac{2N}{N-2}$
and with the help of suitable approximations and 
variational methods they proved the existence of 
multiple solutions depending on the values of
$\la.$
Das and Bal
\cite{bal2024multiplicitysolutionsmixedlocalnonlocal}, 
generalized the results for general
$p \in (1,N).$
In addition, a nonexistence result is 
obtained for sufficiently large
$\la.$
For critical perturbation, Biagi and Vecchi
\cite{Biagi} followed the approach of Haitao
\cite{Haito} and Hirano
\cite{hirano_2004} to obtain the 
existence of multiple solutions.
However, the presence of the fractional 
Laplacian together with the Laplacian 
creates a lack of scaling invariance of 
the mixed operator.
To deal with this, they introduced a 
sufficiently small parameter
$\ve$
with the fractional Laplacian.

Inspired by the literature above,
we are interested in studying the problem
\eqref{eq_constant} in the variable exponent setup,
which leads to the problem
\eqref{eq1}.
This problem can be treated as an extension of 
the problem studied by Dhanya \textit{et al.} in
\cite{dhanya2025multiplicityresultsmixedlocal}.
The authors have considered the problem
\eqref{eq_constant} with right-hand term of the form
$$
\la\big(a(x) |u|^{\delta -2}u + b(x) |u|^{r-2}u\big),
$$
where
$\delta$
corresponds to the sublinear term and
$r$
has superlinear growth that may take critical value
$(p^*$
or
$q^*_s).$
They used the Nehari manifold technique to obtain the 
existence of at least two nonnegative solutions for 
sufficiently small values of
$\la.$
It is worth mentioning that the energy functional associated 
with \eqref{eq1} is not differentiable, even in the sense of 
G\^ateaux derivative, which motivates us to study the problem 
using the Nehari Manifold method.
To the best of our knowledge, the study of mixed local and 
nonlocal operators involving variable exponents has not been 
previously addressed, even in the absence of singular terms.
We established an appropriate functional framework for 
exploring problems involving mixed variable exponent 
operators.
The presence of variable exponents presents another 
difficulty.
To overcome this, we carefully used degree theory arguments 
in the fibering maps analysis.
We find that the combined effects of singular and 
superlinear 
nonlinearities considerably change the structure of the 
solution set, which leads to the existence of multiple 
solutions.
In addition, we established a uniform bound of the weak 
solutions with the De Giorgi iteration technique and a 
localization argument.

The paper is organized as follows.
In Section \ref{sec1}, we introduce the 
notations and preliminary results required 
for the rest of the paper and state our main results.
In Section \ref{sec2}, we study the fibering maps 
associated with the energy functional on subsets of 
the Nehari manifold.
In addition, we analyse the structure of the Nehari 
manifold in detail and establish several related results.
Finally, in Section \ref{sec3}, we prove our main 
results by showing that the minimizers on these 
subsets of the Nehari manifold correspond to 
solutions of the problem.

\section{Preliminaries and Main Results}\label{sec1}
In this section, we recall some definitions 
and preliminary results related to variable 
exponents and associated function spaces.
We also present the statement of our main 
result and formulate the functional setting 
required to study \eqref{eq1}.
We start by introducing the following notation: 
for a function
$h(x)\in C(\oline\om,\real),$
we write
$$
h^+:=\sup_{x\in\oline\om}h(x), 
\qquad h^-:=\inf_{x\in \oline\om}h(x).
$$
We set 
$$
\Phi_+:=\max\{\Phi,0\}, \qquad
 \Phi_-:=\max\{-\Phi,0\}.
$$

Let us define the function space
$$
\mathcal{C}_+(\oline\om):=
\Big \{h\in C(\oline\om,\R):1<
\inf_{x\in\oline\om}h(x)\le
\sup_{x\in \oline\om}h(x)<\infty
\Big \}.
$$
The variable exponents
$\ga(x),r(x)$
and the weights
$a(x),b(x)$
verify the following assumptions:
\begin{enumerate}

\item[$(A_1)$]
$a:\om \ra \real$
such that
$$
0<a\in L^{
\frac{p^*(x)}{p^*(x)-(1-\ga(x))}}(\om),
$$
where
$$
p^*(x):=
\frac{Np(x,x)}{N-p(x,x)}
$$
is the critical exponent.

\item[$(B_1)$]
$b:\om \ra \real$
such that
$$
0\le b\in L^{
\frac{p^*(x)}{p^*(x)-r(x)}}(\om) .
$$

\item[$(G_1)$] The exponent
$\ga(x)\in C(\oline\om,\real)$
and
$0<\ga^-\le \ga(x)\le \ga^+<1$.
    
\item[$(G_2)$] The exponent
$r(x)\in C(\oline\om,\real)$
and
$$
1<p^-\le p(x)\le p^+<r^-\le r(x)\le r^+<(p^*)^- .
$$
    
\item[$(G_3)$] We have that
    \begin{equation}\label{eq5}
\frac{p^+-(1-\ga^+)}{r^+-(1-\ga^+)}p^-<
\frac{p^--(1-\ga^-)}{r^+-(1-\ga^-)}r^+.
    \end{equation}
\end{enumerate}

\begin{remark}\rm
 We make the following observations
\begin{enumerate}
    
\item It is easy to see that if all the 
    exponents are constant, condition
$(G_3)$
becomes
$0<p<r.$
    
\item  One possible choice of variable 
    exponents that satisfies assumptions
$(G_1)-(G_3)$
is given by
$\gamma(x) = 
\frac{|x|^2+1}{3}, p(x) = 
\frac{|x|^2+4}{3},$
and
$r(x) = 
\frac{|x|^2+7}{3},$
here
$x\in B_1(0)\subset\rnn$
\end{enumerate}    
\end{remark}

Now, we formally state the main results of this paper.

\begin{theorem}\label{main_theorem}
    Let
$\Omega \subset \mathbb{R}^N$
be a bounded domain with Lipschitz boundary and let
$s \in (0,1).$
Assume that the function
$p(x, y)$
satisfies conditions
$(P_1)$
and
$(P_2),$
and that the assumptions
$(A_1),$
$(B_1),$
and
$(G_1)-(G_3)$
hold.
Then, there exists a positive constant
$\Lambda > 0$
such that for any
$\lambda \in (0, \Lambda),$
problem \eqref{eq1} admits at least two positive 
weak solutions.
\end{theorem}

Next, we state the following regularity result.

\begin{theorem}\label{reg_result}
 Let
$b(x)\in L^\infty(\om)$
and
$u$
be a positive weak solution of problem \eqref{eq1}.
Suppose that one of the following conditions holds
\begin{enumerate}
    
\item[\textnormal{(a)}]
$a(x) \in L^\infty(\Omega)$; or
    
\item[\textnormal{(b)}]
$a(x) \in L^{
\frac{r(x)}{r(x) - (1 - \ga(x))}}(\Omega)$
and the variable exponents
$r(x),$
$p(x)$
satisfy
    \begin{enumerate}
        
\item[$(i)$]
$\dfrac{r^+}{(p^-)^*} < \dfrac{(r^-)^2}{r^+p^-} 
< 1 + \dfrac{r^+}{(p^-)^*}$;
        
\item[$(ii)$]
$\dfrac{(r^+)^2}{(p^-)^*} < \dfrac{r^-}{p^-}
< 1+\dfrac{(r^+)^2}{(p^-)^*}.$
    \end{enumerate}
\end{enumerate}
 Then
$u$
is of class
$L^\infty(\om)$.
\end{theorem}

\begin{remark}\rm
Let
$B_1(0) \subset \R^3$
and consider the exponents
$p(x) = 
\frac{15 + 2|x|^2}{10}$
and
$r(x) = 
\frac{90 + 3|x|^2}{50}.$
It can be verified that these choices of
$p(x)$
and
$r(x)$
satisfy the assumptions of Theorem~\ref{reg_result}.
Moreover, by taking
$\gamma(x) = 
\frac{11 - 2|x|^2}{20},$
the conditions
$(G_1)$–$(G_3)$
are also fulfilled.
\end{remark}

 In order to study the positive solutions of 
\eqref{eq1}, we consider the following problem
\begin{equation}\label{eq2}\tag{$\mathcal{Q}^+_{\lambda}$}
   \left\{
	\begin{aligned}
		-\De_{p(x)}u+\lb -\De \rb^s_{p(x)}u 
&=\lambda a(x)u_+^{-\gamma(x)}+b(x)u_+^{r(x)-1},\,
 u>0 \  \text{ in } \om,\\
u &= 0 \qquad \text{ in } \rnn\backslash\om.
	\end{aligned}
    \right.
\end{equation}
It is easy to observe that if
$u>0$
is a weak solution of \eqref{eq2}, then
$u$
also satisfies the problem \eqref{eq1}.
To study the problem \eqref{eq2} in the variational setup, 
we require the following variable exponent function spaces.
We start with the definition of the variable Lebesgue space
$L^{p(x)}(\om).$
For
$p\in \mathcal{C}_+(\overline{\Omega}),$
define 
\begin{equation*}
    L^{p(x)}(\om) = 
\Big \{ u : \om \to \real : u  
    \text{ is measurable and } 
\int_\om|u(x)|^{p(x)}dx < \infty 
\Big \}.
\end{equation*}
This is a Banach space endowed with the norm, 
known as the Luxemburg norm, given by
\begin{equation*}
    \lv u \rv_{L^{p(x)}} = \inf
\Big \{ \eta > 0  : 
\int_{\om}
\Big | 
\frac{u(x)}{\eta} 
\Big |^{p(x)} dx \le 1 
\Big \}.
\end{equation*}
Moreover, it satisfies the following 
H\"{o}lder-type inequality analogous to 
that of the classical
$L^p{(\om)}$
spaces.

\begin{lemma}
	Let
$p\in \mathcal{C}_+(\overline{\Omega})$
such that
$ 
\frac{1}{p(x)} + 
\frac {1}{p'(x)}=1$.
	Then for any 
$u \in L^{p(x)}(\Omega)$
and
$
v\in L^{p'(x)}(\Omega)$
we have
$$
\Big | 
\int_{\Omega} uv dx 
\Big |\leq \Big(
\frac{1}{p^{-}} + 
\frac{1}{p^{'-}} \Big) \lv u \rv_{L^{p(x)}(\Omega)}
	\lv v \rv_{L^{p'(x)}(\Omega)}.
$$
\end{lemma}

 We recall a useful inequality 
that involves two variable exponents.

\begin{lemma}[Lemma $A.1$
of \cite{giacomoni_para_px}]\label{lemA1}
	Let
$\nu_1(x)\in L^\infty(\Omega)$
such that
$\nu_1\geq0,\; \nu_1\not\equiv 0.$
Let
$\nu_2:\Omega\ra\real$
be a measurable function 
	such that
$\nu_1(x)\nu_2(x)\geq 1$
a.e. in
$\Omega.$
Then for every
$u\in L^{\nu_1(x)\nu_2(x)}(\Omega),$
$$
	\lv |u|^{\nu_1(x)}\rv_{L^{\nu_2(x)}(\Omega)}\leq 
	\lv u\rv_{L^{\nu_1(x)\nu_2(x)}
(\Omega)}^{\nu_1^-}+\lv u
\rv_{L^{\nu_1(x)\nu_2(x)}(\Omega)}^{\nu_1^+}.
$$
\end{lemma}	

To handle the Luxemburg norm, 
we define the modular function
$\rho:L^{p(x)}(\Omega)\ra \real$
given as
$$
\rho(u) = 
\int_\Omega |u|^{p(x)}dx.
$$
The relations between Luxemburg norm
$\lv \cdot \rv_{L^{p(x)}(\Omega)}$
and 
the corresponding modular function
$\rho(\cdot)$
are given as follows:

\begin{lemma}[\cite{fan_lp_wp_spaces}]
\label{modular_ineq}
	Let
$u \in L^{p(x)}(\Omega),$
then 
	\begin{enumerate}
		
\item[$(i)$]$
\lv u \rv_{L^{p(x)}(\Omega)}<1(=1;>1) 
\text{ if and only if } \rho(u)<1(=1;>1);$
		
\item[$(ii)$] If
$\lv u \rv_{L^{p(x)}(\Omega)}>1,$
then
$\lv u \rv_{L^{p(x)}(\Omega)} ^{p^{-}}
\leq\rho(u)\leq\lv u \rv_{L^{p(x)}(\Omega)}^{p^{+}}$;

\item[$(iii)$] If
$\lv u \rv_{L^{p(x)}(\Omega)}<1,$
then
$\lv u \rv_{L^{p(x)}(\Omega)}^{p^{+}}
\leq\rho(u)\leq\lv u \rv_{L^{p(x)}(\Omega)}^{p^{-}}.$
	\end{enumerate}
\end{lemma}

\begin{lemma}[\cite{fan_lp_wp_spaces}]
\label{lem:modular_cgs_lp}
	Let
$u,u_{m}  \in L^{{p(x)}}(\Omega),~m=1,2,3,\cdots .$
Then the following statements are equivalent:
	\begin{enumerate}
		
\item[$(i)$]
$\displaystyle{\lim_{m\ra \infty} }\lv 
u_{m} - u \rv_{L^{p(x)}} =0;$
		
\item[$(ii)$]
$\displaystyle{\lim_{m\ra \infty}} 
\rho(u_{m} -u)=0;$
		
\item[$(iii)$]
$
u_{m} \text{ converges to
$u$
in }  \Omega  \text { in measure and } 
\displaystyle{\lim_{m\ra \infty}} \rho(u_{m})
= \rho(u).$
	\end{enumerate}
\end{lemma}

Let
$p:\oline\om\to \R$
be such that
$1< p^-\le p(x)\le p^+.$
We say that
$p$
is log-H\"older continuous 
(see \cite{fan_lp_wp_spaces}) 
if there exists a positive constant
$M$
such that
$$
|p(x)-p(y)|\le
\frac{M}{-\log|x-y|}, 
\text{ for all }x,y\in\oline\om,\,\text{with } |x-y|<
\frac{1}{2}.
$$

In what follows, we assume that
$p$
is
$\log$-Hölder continuous.
Now, with the help of the variable Lebesgue space, 
the variable Sobolev space
$W^{1,p(x)}(\om)$
is defined as
\begin{equation*}
    W^{1,p(x)}(\om) =
\Big \{ u \in L^{p(x)}(\om) : 
    |\grad u|^{p(x)} \in L^{p(x)}(\om) 
\Big \},
\end{equation*}
with the norm
$$
\lv u \rv_{1,p(x)} = \lv u \rv_{L^{p(x)}(\om)} 
+ \lv \grad u \rv_{L^{p(x)}(\om)}.
$$
The space
$W^{1,p(x)}_0(\om)$
is defined as the closure of
$C_0^\infty(\om)$
in
$W^{1,p(x)}(\om).$

The following result is a Sobolev embedding 
theorem for variable exponent spaces; its 
proof can be found in 
\cite{Variable_book} and
\cite{fan_lp_wp_spaces}.

\begin{lemma}\label{lm:Sobolev_embedding}
    Let
$\om\subset\rnn$
be a bounded open set with Lipschitz boundary.
If
$p$
is
$\log$-H\"older continuous then 
    \begin{equation*}
        W^{1,p(x)}(\om)
\hookrightarrow L^{q(x)}(\om) 
\quad\text{ for any }q(x)\in L^\infty(\om)
    \end{equation*}
    such that
$q(x)$
satisfies
$1<q(x)\le p^*(x)$
for all
$x\in\om.$
Also, this embedding is compact if
$1<q(x)< p^*(x).$
\end{lemma}

For
$s\in(0,1)$
and a function
$p \in C(\om \times \om)$
with the symmetric property, i.e.,
$p(x,y) = p(y,x)$
for all
$x,y \in \om,$
the fractional Sobolev space with variable exponents
$W^{s,p(x),p(x,y)}(\om)$
is defined as 
\begin{align*}
 \mathcal{W}
 & =W^{s,p(x),p(x,y)}(\om) \\
 & = 
 \!
\Big \{ u \in L^{p(x)}(\om) : 
\!
\int_{\om \times \om} 
\frac{|u(x) - u(y)|^{p(x,y)}}{\eta^{p(x,y)} 
|x-y|^{N + s p(x,y)}} dy \, dx < \! \infty 
\text{ for some } \eta \! >0
\Big \}.
\end{align*}
It has a seminorm as
\begin{equation*}
    [u]_{\om}^{s,p(x,y)} = \inf 
\Big \{ \eta > 0 : 
\int_{\om \times \om} 
\frac{|u(x) - u(y)|^{p(x,y)}}{\eta^{p(x,y)} 
|x-y|^{N + s p(x,y)}} dy \, dx < 1 
\Big \}.
\end{equation*}
This becomes a Banach space when endowed with 
the following norm 
\begin{equation*}
    \lv u \rv_{\mathcal{W}} := 
    \lv u \rv_{L^{p(x)}} 
    + [u]_{\om}^{s,p(x,y)}.
\end{equation*}
We now state the following theorem on 
continuous and compact embeddings related 
to the fractional Sobolev space with variable exponents.

\begin{lemma}[\cite{KyHo}]
    Let
$\om$
be a bounded Lipschitz domain in
$\rnn.$
Let
$s\in(0,1)$
and
$p(x,y)\in \mathcal{C}_+(\rnn\times\rnn)$
satisfying
$p(x,y)=p(y,x)$
for all
$x,y\in \rnn$
with
$sp^+<N.$
Let
$q\in \mathcal{C}_+(\oline\om)$
such that
$$
1<q(x)<p^*_s(x):=
\frac{Np(x,x)}{N-sp(x,x)}
$$
for all
$x\in\om.$
Then for all
$u\in\mathcal{W}$
$$
\|u\|_{L^{q(x)}}\le C(N,s,p,q,\om)\|u\|_{\mathcal{W}}.
$$
    Moreover, this embedding is also compact.
\end{lemma}

\begin{remark}\rm
We note that the continuous embedding remains 
valid in the critical case 
(i.e.,
$q(x)=p^*_s(x)$), provided certain additional 
conditions are imposed on the exponents
$p(x,y)$
and
$q(x).$
For more details, we refer to
\cite{KyHo_critical}.
\end{remark}

Given that problem 
\eqref{eq1} involves the nonlocal operator
$(-\Delta)^s_{p(x)}$
with Dirichlet boundary conditions, 
we define another fractional Sobolev-type space 
with variable exponents to study the existence 
of solutions via variational analysis.
Let us denote the set
$\Q:=\R^{2\N}\sm(\om^c\times\om^c)$
and define the following space
\begin{align*}
 \mathbb{X} & =X^{s,p(x),p(x,y)}(\om) \\
& =  \Big \{ u:\rnn \to \R :u|_\om 
 \in L^{p(x)}(\om),  \\
&
\qquad 
\int_{\Q} 
\frac{|u(x) - u(y)|^{p(x,y)}}
{\eta^{p(x,y)} |x-y|^{N + s p(x,y)}} dy \, dx 
< \infty \text{ for some } \eta >0 \Big \},
\end{align*}
which is equipped with the norm 
\begin{equation*}
    \lv u \rv_{\mathbb{X}} := 
    \lv u \rv_{L^{p(x)}} + [u]_{\mathbb{X}}^{s,p(x,y)},
\end{equation*}
where
$[u]_{\mathbb{X}}^{s,p(x,y)}$
is the seminorm given by
\begin{equation*}
    [u]_{\mathbb{X}}^{s,p(x,y)} = 
    \inf 
\Big \{ \eta > 0 : 
\int_{\Q} 
\frac{|u(x) - u(y)|^{p(x,y)}}{\eta^{p(x,y)} 
|x-y|^{N + s p(x,y)}} dy \, dx < 1 
\Big \}.
\end{equation*}
The space
$(\mathbb{X},\|\cdot\|_\mathbb{X})$
is separable Banach space.
We define the following subspace of
$\mathbb{X}$
$$
\mathbb{X}_0=X^{s,p(x),p(x,y)}_0(\om):=
\Big \{u\in \mathbb{X}:u=0 \,\text{ a.e. in }\,\om^c
\Big \}.
$$
Norm on
$\mathbb{X}_0$
is defined as follows
\begin{equation*}
    \|u\|_{\mathbb{X}_0} = \inf
\Big \{ \eta > 0 : 
\int_{\Q} 
\frac{|u(x) - u(y)|^{p(x,y)}}
{\eta^{p(x,y)} |x-y|^{N + s p(x,y)}} dy \, dx 
< 1 
\Big \}.
\end{equation*}

\begin{remark}\rm
    We can observe that for
$u\in\mathbb{X}_0$
    \begin{equation*}
\int_{\rnn\times\rnn} 
\frac{|u(x) - u(y)|^{p(x,y)}}{\eta^{p(x,y)} 
|x-y|^{N + s p(x,y)}} dy \, dx=
\int_{\Q} 
\frac{|u(x) - u(y)|^{p(x,y)}}{\eta^{p(x,y)} 
|x-y|^{N + s p(x,y)}} dy \, dx.
\end{equation*}
\end{remark}

Therefore, we can consider the following norm on
$\mathbb{X}_0$
\begin{equation*}
    \|u\|_{\mathbb{X}_0} = \inf
    \Big \{ \eta > 0 : 
\int_{\rnn\times\rnn} 
\frac{|u(x) - u(y)|^{p(x,y)}}{\eta^{p(x,y)} 
|x-y|^{N + s p(x,y)}} dy \, dx < 1 
\Big \}.
\end{equation*}
We define the modular function
$\rho_{\mathbb{X}_0}:L^{p(x)}(\Omega)\ra \real$
given as
$$\rho_{\mathbb{X}_0}(u) = 
\int_{\rnn\times\rnn} 
\frac{|u(x) - u(y)|^{p(x,y)}}
{|x-y|^{N + s p(x,y)}} dy \, dx.$$
The relations between Luxemburg norm
$\lv \cdot \rv_{\mathbb{X}_0}$
and 
the corresponding modular function
$\rho_{\mathbb{X}_0}(\cdot)$
are given as follows:

\begin{lemma}[\cite{KyHo}]
	Let
$u \in \mathbb{X}_0,$
then the following statements hold
	\begin{enumerate}
		
\item[(i)]$
\lv u \rv_{\mathbb{X}_0}<1(=1;>1) 
\text{ if and only if } 
\rho_{\mathbb{X}_0}(u)<1(=1;>1);$

\item[(ii)] If $\lv u \rv_{\mathbb{X}_0}>1,$
then $\lv u \rv_{\mathbb{X}_0} ^{p^{-}}
\leq\rho_{\mathbb{X}_0}(u)\leq\lv u
 \rv_{\mathbb{X}_0}^{p^{+}}$;
		
\item[(iii)] 
		If 
 $\lv u \rv_{\mathbb{X}_0}<1,$
then 
$\lv u \rv_{\mathbb{X}_0}^{p^{+}}
\leq
\rho_{\mathbb{X}_0}(u)\leq\lv u 
\rv_{\mathbb{X}_0}^{p^{-}}.$
	\end{enumerate}
\end{lemma}

\begin{lemma}[\cite{KyHo}]
	Let 
$u,u_{m}  \in \mathbb{X}_0,~m=1,2,3,\cdots .$
Then the following statements are equivalent
	\begin{enumerate}
	
\item[(i)] $\displaystyle{\lim_{m\ra \infty} }
\lv u_{m} - u \rv_{\mathbb{X}_0} =0;$

\item[(ii)] $\displaystyle{\lim_{m\ra \infty}} 
\rho_{\mathbb{X}_0}(u_{m} -u)=0.$
	\end{enumerate}
\end{lemma}

\begin{lemma}[\cite{fract_extended}]
    The space $(\mathbb{X}_0,\|\cdot\|_{\mathbb{X}_0})$
is separable, reflexive and uniformly convex Banach space.
\end{lemma}

In order to study \eqref{eq1}, 
which involves a mixed operator with 
variable exponents, the `natural' 
functional framework is provided by the following space:
$$
\mathcal{X}(\om)=W^{1,p(x)}(\om)\cap\mathbb{X},
$$
which is equipped with the norm
$$
\|u\|_{\mathcal{X}}:=\|u\|_{1,p(x)}+\|u\|_\mathbb{X}.
$$
Since problem \eqref{eq1} involves Dirichlet boundary 
conditions, we define the corresponding function space
\begin{equation}\label{eqs1}
    \X(\om)=W^{1,p(x)}_0(\om)\cap\mathbb{X}_0,
\end{equation}
endowed with the norm
\begin{equation}\label{norm}
\|u\|_{\mathcal{X}_0}:=
\|u\|_{1,p(x)}+\|u\|_{\mathbb{X}_0}.
\end{equation}
From Lemma \ref{lm:Sobolev_embedding} and 
equation \eqref{norm}, we can state the 
following embedding result.

\begin{lemma}\label{lm:space_embedd}
    Let $\om\subset\rnn$
be a bounded domain with Lipschitz boundary.
Then we have
    \begin{equation*}
        \|u\|_{L^{q(x)}(\om)}\le \|u\|_{\X},\,\,
\forall u\in\X \text{ and for any }q(x)\in L^\infty(\om)
    \end{equation*}
    such that $1<q(x)\le p^*(x)$
for all $x\in\om.$
Moreover, the embedding $\X\hookrightarrow L^{q(x)}(\om)$
is continuous, and it is compact if $1<q(x)< p^*(x).$
\end{lemma}

\begin{remark}\rm
    It is well known that $W^{1,p}(\rnn)$
is continuously embedded in $W^{s,p}(\rnn)$ (see
\cite[Lemma 2.1]{embedd_local_to_nonlocal}).
Therefore, while dealing with equations involving 
mixed operators, it is natural to consider 
the solution space
$$
\mathbb{W}(\om)=\{u\in W^{1,p}(\rnn):u=0 \,
\text{ a.e. in }\,\om^c\}.
$$
However, in the variable exponent setting, 
it is not clear whether 
$W^{1,p(x)}(\rnn)$
is continuously embedded in 
$W^{s,p(x,y)}(\rnn).$
For this reason, 
we define the solution space as in \eqref{eqs1}.
\end{remark}

\begin{definition}
    For $u\in\X,$
we define the modular function $\rho_{\X}:\X\to \R$
as
$$
    \rho_{\X}(u) = 
\int_\om|\na u|^{p(x)}+\iint_{\R^{2N}} 
\frac{|u(x) - u(y)|^{p(x,y)}}
{|x-y|^{N + s p(x,y)}} dy \, dx.
$$
\end{definition}

From
\cite[Definition 2.1.1]{Variable_book}, 
we can verify that $\rho_{\X}$
is a modular on $\X$
and satisfy the following relations:

\begin{lemma}
	Let $u \in \X.$
Then, the following properties hold
	\begin{enumerate}
		
\item[(i)]$
\lv u \rv_{\X}<1(=1;>1) \text{ if and only if }
 \rho_{\X}(u)<1(=1;>1);$
		
\item[(ii)] If $\lv u \rv_{\X}>1,$
then $\lv u \rv_{\X} ^{p^{-}}\leq\rho_{\X}(u)
\leq\lv u \rv_{\X}^{p^+}$;

\item[(iii)] If $\lv u \rv_{\X}<1,$
then $\lv u \rv_{\X}^{p^{+}}\leq\rho_{\X}(u)
\leq\lv u \rv_{\X}^{p^{-}}.$
	\end{enumerate}
\end{lemma}

\begin{lemma}\label{lem:modular_cgs_prob}
	Let 
$u,u_{m}  \in \X,~m=1,2,3,\cdots .$
Then the following statements are 
equivalent
	\begin{enumerate}

\item[(i)] $\displaystyle{\lim_{m\ra \infty} }\lv
 u_{m} - u \rv_{\X} =0;$

\item[(ii)] $\displaystyle{\lim_{m\ra \infty}}
 \rho_{\X}(u_{m} -u)=0.$
	\end{enumerate}
\end{lemma}

\begin{lemma}
    The space $(\X,\|\cdot\|_{\X})$
is a separable, reflexive and uniformly 
convex Banach space.
\end{lemma}

\begin{remark}\rm
    An equivalent norm on $\X$
can be defined using the modular function as
   \begin{equation*}
    \lv u \rv: = \inf
\Big \{ \eta > 0  : \rh
\Big (
\frac{u}{\eta}
\Big ) \le 1 
\Big \}
\end{equation*}
    with the relation 
    $$
\frac{1}{2}\|u\|_{\X}\le \|u\|\le\|u\|_{\X},
 \quad \forall u\in\X.$$
\end{remark}

Now, we define the notion of a weak solution
 for the problem \eqref{eq1}.
 
\begin{definition}
    We say that a function $u\in \X$
is a weak solution to $\eqref{eq1}$
if $u>0$
in $\om$
and
\begin{align}\label{weak_sol}
&
\int_\om|\grad u|^{p(x)-2}\na u\na\phi\, dx  \\ \notag
  + 
& \iint_{\mathbb{R}^{2N}} 
\frac{|u(x)-u(y)|^{p(x,y)-2}(u(x)-u(y))
(\phi(x)-\phi(y))}{|x-y|^{N+sp(x,y)}}\,dxdy
  \\ \notag
 -\la 
 &
\int_\om a(x)u^{-\ga(x)}\phi\,dx -
\int_\om b(x)u^{r(x)-1}\phi\,dx=0
\  \text{ for all  $ \,\phi \in \X$.}
	\end{align}
 
\end{definition}

To handle certain maps involving variable 
exponent powers, it is useful to approximate 
them by functions with constant exponents.
For this purpose, we use the following notion 
of topological degree, as developed in
\cite{deimling_degree_theory}.

\begin{definition}\label{def:degree}
The degree of a function is a function 
\begin{align*}
\deg : \Big\{
& (f, \Omega, y) : 
\Omega \subset \mathbb{R}^n 
\text{ open and bounded}, \\
&  
\qquad
f : \overline{\Omega} 
\to \mathbb{R}^n\ \text{continuous}, y \in 
\mathbb{R}^n \setminus f(\partial \Omega) \Big\}
 \to \mathbb{Z}
\end{align*}
that satisfies the following properties :

\begin{itemize}

\item[$(d1)$] {Identity:} For the identity map 
$id: \Omega \to \Omega,$
we have
    \[
    \deg(\text{id}, \Omega, y) = 1 
  \quad \text{for } y \in \Omega.
    \]
    
\item[$(d2)$] {Additivity:}
    \[
    \deg(f, \Omega, y) = \deg(f, \Omega_1, y) 
 + \deg(f, \Omega_2, y)
    \]
    whenever $\Omega_1, \Omega_2$
are disjoint open subsets of $\Omega$
such that 
$$
y \notin f
\left (\overline{\Omega} \setminus 
(\Omega_1 \cup \Omega_2)
\right ). 
$$
    
\item[$(d3)$] {Homotopy Invariance:}
    \[
    \deg(h(t, \cdot), \Omega, y(t)) 
   \text{ is independent of } t \in  [0, 1]
    \]
    whenever 
$h : [0,1] \times \overline{\Omega} \to \mathbb{R}^n$
is continuous,
$y : [0,1] \to \mathbb{R}^n$
is continuous, and the function
$y(t) \notin h(t, \partial \Omega)$
for all $t \in [0,1]$.
\end{itemize}
\end{definition}

Note that there is a unique function 
that satisfies the aforementioned properties.
The homotopy invariance properties are 
the most crucial property that leads 
to the following useful lemma. 

\begin{lemma}\label{lem:equal_degree}
Let $f \in C(\overline{\om})$
with $y \not \in f(\partial \om)$
and $d = dist(y, f(\partial \om)).$
If $g_1, g_2 \in C^{2}(\overline{\om})$
such that 
$$
\lv g_i - f\rv_{L^{\infty}(\om)} < d 
\ \text{ for $i = 1,2$;}
$$
 then 
$$
 \deg(g_1, \om,y) = \deg(g_2, \om,y) .
 $$
\end{lemma}

 The above lemma follows from
\cite[Proposition $2.4$]{deimling_degree_theory}.

\section{Nehari Manifold and Fiber Map Analysis}\label{sec2}

The energy functional $J_{\lambda}: \X \to \real$
associated with problem \eqref{eq2} is given by:
\begin{align}\label{eq:eneryg_functional}
J_\la(u)=
\int_\om 
\frac{1}{p(x)}|\grad u|^{p(x)}\,dx
&+\iint_{\mathbb{R}^{2N}} 
\frac{1}{p(x,y)}
\frac{|u(x)-u(y)|^{p(x,y)}}{|x-y|^{N+sp(x,y)}}dx \, dy \\ \notag
-\la&
\int_\om 
\frac{a(x)}{1-\ga(x)}(u_+)^{1-\ga(x)}\,dx-
\int_\om 
\frac{b(x)}{r(x)}(u_+)^{r(x)}\,dx.
\end{align}
It is important to note that the functional $J_\la$
is well defined; however, it is not 
differentiable due to the presence of the singular term.

 For \( u \in \X \), we define the 
 fibering map \( \phi_u(t):\R^+\to \R \) as
\begin{equation}\label{eq:fibering_map}
\begin{aligned}
\phi_u(t) 
&= 
\int_{\Omega} 
\frac{t^{p(x)}}{p(x)} |\nabla u|^{p(x)} \, dx 
+ \iint_{\mathbb{R}^{2N}} 
\frac{t^{p(x,y)}}{p(x,y)} 
\frac{|u(x) - u(y)|^{p(x,y)}}{|x - y|^{N + sp(x,y)}}
 \, dx \, dy \\
& - \lambda 
\int_{\Omega} 
\frac{t^{1 - \gamma(x)}}{1 - \gamma(x)}
 a(x) (u_+)^{1 - \gamma(x)} \, dx 
- 
\int_{\Omega} 
\frac{t^{r(x)}}{r(x)} b(x) (u_+)^{r(x)} \, dx.
\end{aligned}
\end{equation}
The first derivative of the fibering map is given by
\begin{align} \label{eq:fibering_map_first_derivative}
\phi_u'(t) 
& = 
\int_{\Omega} t^{p(x)-1} |\nabla u|^{p(x)} \, dx 
+ \iint_{\mathbb{R}^{2N}} t^{p(x,y)-1} 
\frac{|u(x) - u(y)|^{p(x,y)}}{|x - y|^{N + sp(x,y)}}
 \, dx \, dy  \notag \\
&  - \lambda 
\int_{\Omega} t^{-\gamma(x)} a(x) (u_+)^{1 - \gamma(x)} 
\, dx 
- 
\int_{\Omega} t^{r(x)-1} b(x) (u_+)^{r(x)} \, dx.
\end{align}
The second derivative is given by
\begin{align}\label{eq:fibering_map_second_derivative}
 \phi_u''(t) 
& = 
\int_{\Omega} (p(x) - 1) t^{p(x)-2} |\nabla u|^{p(x)}
 \, dx  \\ \notag
& + \iint_{\mathbb{R}^{2N}} (p(x,y) - 1) t^{p(x,y)-2} 
\frac{|u(x) - u(y)|^{p(x,y)}}{|x - y|^{N + sp(x,y)}} 
\, dx \, dy   \\ \notag
&  + \lambda 
\int_{\Omega} \gamma(x) a(x) 
t^{-\gamma(x)-1} (u_+)^{1 - \gamma(x)} \, dx   \\ \notag
&
- 
\int_{\Omega} (r(x) - 1) b(x) t^{r(x)-2}
 (u_+)^{r(x)} \, dx.
\end{align}
We note that the energy functional 
\( J_{\lambda} \) is not bounded below on 
\( \X \).
To address this difficulty, we restrict our 
attention to the following subset of \( \X \), 
known as the Nehari manifold, defined by
\begin{equation} \label{eq:nehari_manifold}
\mathcal{N}_{\lambda} = 
\Big \{ u \in \X\sm\{0\}
 \;:\phi_u'(1)= 0
 \Big \}.
\end{equation}
Thus, $tu\in \Nh$
if and only if $\phi'_u(t)=0.$
In particular, $u\in \Nh$
if and only if $\phi_u'(1)=0.$

For $u\in \Nh,$
from \eqref{eq:fibering_map}, 
\eqref{eq:fibering_map_first_derivative} and 
\eqref{eq:fibering_map_second_derivative}, we get
\begin{align} \label{eq:first_derivative} 
\phi_u'(1) 
& = 
\int_{\Omega}|\nabla u|^{p(x)} \, dx 
 +\iint_{\mathbb{R}^{2N}}
\frac{|u(x) - u(y)|^{p(x,y)}}{|x - y|^{N + sp(x,y)}}
 \, dx \, dy   \\ \notag
&  - \lambda 
\int_{\Omega} a(x) (u_+)^{1 - \gamma(x)} \, dx 
- 
\int_{\Omega}b(x) (u_+)^{r(x)} \, dx
\end{align}
and 
\begin{align}\label{eq:second_derivative} 
\phi_u''(1) 
& = 
\int_{\Omega} p(x)|\nabla u|^{p(x)}  \, dx 
+ \iint_{\mathbb{R}^{2N}} p(x,y)
\frac{|u(x) - u(y)|^{p(x,y)}}
{|x - y|^{N + sp(x,y)}} \, dx \, dy   \\ \notag
& - \lambda 
\int_{\Omega} (1-\gamma(x)) 
a(x)(u_+)^{1 - \gamma(x)} \, dx 
- 
\int_{\Omega} r(x)b(x)(u_+)^{r(x)} \, dx.
\end{align}
We note that $u\in\Nh$
if $u$
is a weak solution of \eqref{eq2}.
In view of the structure of the fibering map, 
it is natural to decompose the Nehari 
manifold into the following three sub-parts 
\begin{align*}
    \mathcal{N}_{\lambda}^{0} 
    &:= \left\{ u \in \mathcal{N}_{\la} \;:\; 
    \phi^{\prime \prime}_{u}(1) = 0 \right\}.\\
    \mathcal{N}_{\lambda}^{+} 
    &:= \left\{ u \in \mathcal{N}_{\la} \;:\; 
    \phi^{\prime \prime}_{u}(1) > 0 \right\}.\\
    \mathcal{N}_{\lambda}^{-} 
    &:= \left\{ u \in \mathcal{N}_{\la} \;:\; 
    \phi^{\prime \prime}_{u}(1) < 0 \right\}.
\end{align*}

 Now, let us write
\begin{equation}\label{fibermap_derivative_f_i}
\phi'_u(t) = f_1(t) - \lambda f_2(t) - f_3(t) ,
\end{equation}
where
\begin{align}\label{functions}
 &  f_1(t) = \ds
\int_{\Omega} t^{p(x)-1} |\nabla u|^{p(x)} \, dx 
+ \iint_{\mathbb{R}^{2N}} t^{p(x,y)-1}
\frac{|u(x)-u(y)|^{p(x,y)}}
{|x - y|^{N + sp(x,y)}}\, dx \, dy,    \notag \\
&   f_2(t) = \ds
\int_{\Omega} a(x)t^{-\gamma(x)}(u_+)^{1 - \gamma(x)}
 \, dx,   \\ \notag
 &
 f_3(t) = \ds
\int_{\Omega} b(x) t^{r(x)-1} (u_+)^{r(x)} \, dx.
 \end{align}

The functions $f_i$
for $i = 1, 3$
are continuous, strictly increasing, and satisfy 
$f_i(0) = 0.$
To prove some important properties of these functions, 
we approximate them using tools from degree theory.
In this direction, we state the following lemma.

\begin{lemma}
For $u \in \X$
and $0 < \delta_1 < \delta_2,$
let us define the function
$f : (\delta_1, \delta_2) \to \mathbb{R}$
as
\begin{equation*}
f(t) = \iint_{\Q} t^{p(x,y)} h(x,y) \, dx \, dy,
\end{equation*}
where
\begin{equation*}
h(x,y) := 
\frac{|u(x) - u(y)|^{p(x,y)}}
{|x - y|^{N + s p(x,y)}} > 0.
\end{equation*}
Then, for every 
$\varepsilon > 0,$
there exists a disjoint collection 
$\mathcal{B}^{\prime} =  
\bigcup_{i=1}^{\infty} B_i^{\prime}$
of $\Q$
and a corresponding set of exponents 
$\{ p_i \}$
such that, if we define the function 
$g : (\delta_1, \delta_2) \to \mathbb{R}$
by
\begin{equation*} 
g(t) := \sum_{i=1}^\infty t^{p_i} 
\int_{B_i^{\prime}} h(x,y) \, dy \, dx.
\end{equation*}
Then the following holds:
\begin{enumerate}
    
\item $\lv f(t) - g(t) 
\rv_{L^{\infty}(\delta_1, \delta_2)} < \varepsilon,$
    
\item $\deg(f,(\delta_1, \delta_2), 0) 
= \deg(g,(\delta_1, \delta_2), 0) = 0 $,
\end{enumerate}
where $\deg$
denotes the degree as defined in
 Definition $\ref{def:degree}$.
\end{lemma}

\begin{proof}[\bf Proof.]
Define the set
$
A := (\mathbb{Q} \times \mathbb{Q}) \cap \Q.
$
Let \( (x_0, y_0) \in A \) and 
$\varepsilon > 0.$
Then, from the continuity of the map 
\( t^{p(x,y)} \) 
on the set 
\( [\delta_1, \delta_2] \times \Q \), 
and the compactness of 
\( [\delta_1, \delta_2] \), 
there exist a radius 
\( r > 0 \), a neighbourhood 
\( B_r(x_0, y_0) \), 
and a point $p_0 \in [p^-,p^+]$
such that
\begin{equation} \label{eq:uniform_control}
\left| t^{p(x,y)} - t^{p_0} \right| < 
\frac{\varepsilon}{M} \quad 
\text{ for all } (x,y) \in B_r(x_0, y_0)
 \subseteq \mathbb{R}^N 
 \setminus \Omega \times \Omega,  
\end{equation}
  for all 
$t \in [\delta_1, \delta_2],$
where 
$$
M := \ds\iint_{\Q} h(x,y) \, dy \, dx .
$$
Now, consider the collection
\begin{equation*}
\mathcal{B} := \bigcup_{(x_i, y_i) \in A}
 B_{r_i}(x_i, y_i).
\end{equation*}
Since 
\( \mathbb{Q} \times \mathbb{Q} \) 
is countable, we can enumerate the elements of 
\( A \) and write the following disjoint union
\begin{equation*}
\mathcal{B}^{\prime} = 
\bigcup_{i=1}^{\infty} B_i^{\prime},
\end{equation*}
where
$B_1^{\prime} = B_1$
and
$$
B_i^{\prime}= B_{r_i}(x_i,y_i) 
\setminus \bigcup_{j=1}^{i-1} B_{r_j}(x_j,y_j) 
\subset B_{r_i}(x_i,y_i)
$$
such that $\mathcal{B}^{\prime} = \mathcal{B}$.
In addition, the set
\begin{equation*}
\Q  \setminus \mathcal{B}^{\prime} 
= \text{set of isolated points}\,
(\text{denoted by } \mathcal{C}),
\end{equation*}
and since measure of the set 
\( \mathcal{C} = 0 \),
instead of working on 
\( \Q \),
we can equivalently work with
\( \mathcal{B}^{\prime} \).
Now, consider the function
\begin{equation*}
g(t) := \sum_{i=1}^\infty t^{p_i} 
\int_{B_i^{\prime}} h(x,y) \, dy \, dx.
\end{equation*}
Then, we estimate the difference with the 
help of the monotone convergence theorem 
and equation \eqref{eq:uniform_control},
 we deduce
\begin{align*} 
\left| f(t) - g(t) \right| 
 & =
  \Big | \iint_{\Q} t^{p(x,y)} h(x,y) \,
 dy \, dx - g(t) 
  \Big | \\
 & =  
 \Big | \iint_{\mathcal{B}^{\prime}} 
t^{p(x,y)} h(x,y) \, dy \, dx - g(t) 
 \Big | \\
 & = 
  \Big | \sum_{i=1}^{\infty} 
\int_{B_i^{\prime}} t^{p(x,y)} h(x,y) \, 
dy \, dx - \sum_{i=1}^{\infty} t^{p_i} 
\int_{B_i^{\prime}} h(x,y) \, dy \, dx 
 \Big |\\
 & \leq  \sum_{i=1}^\infty 
\int_{B_i^{\prime}} h(x,y) 
\left| t^{p(x,y)} - t^{p_i} \right| \, dy \, dx\\
 & \leq  
\frac{\varepsilon}{M} \sum_{i=1}^{\infty} 
\int_{B_i^{\prime}} h(x,y) \, dy \, dx\\
 & \leq  
\frac{\varepsilon}{M} 
\iint_{\mathcal{B}^{\prime}} h(x,y) \, dy \, dx 
= \varepsilon
\end{align*}
since \( \iint_{\mathcal{B}^{\prime}} h(x,y)
 \, dy \, dx = M \).

Therefore, by Lemma \ref{lem:equal_degree}, 
we conclude that
\begin{equation*} 
\deg
\left ( f, [\delta_1, \delta_2], 0 
\right ) = \deg
\left ( g, [\delta_1, \delta_2], 0 
\right ) = 0.\qedhere
\end{equation*}
\end{proof}
In the spirit of the above lemma and the fact that
 $\overline{\om}$
is compact, we can also deduce the following lemma.

\begin{lemma}\label{lem:function_increasing_degree}
For $u \in \X$
and $0 < \delta_1 < \delta_2,$
let us define the function
$f : (\delta_1, \delta_2) \to \mathbb{R}$
as
\begin{equation*} 
f(t) = 
\int_{\om} t^{p(x)}|\grad u|^{p(x)} dx.
\end{equation*}
Then, for every $\varepsilon > 0,$
there exists a disjoint collection 
$\mathcal{B}^{\prime} = \bigcup_{i=1}^{n} B_i^{\prime}$
and a corresponding set of exponents 
$\{ p_i \}$
such that, if we define the function 
$g : (\delta_1, \delta_2) \to \mathbb{R}$
by 
\begin{equation*} 
g(t) := \sum_{i=1}^n t^{p_i} 
\int_{B_i^{\prime}} |\grad u|^{p(x)}   dx.
\end{equation*}
Then the following holds
\begin{enumerate}
    
\item $\lv f(t) - g(t) 
\rv_{L^{\infty}(\delta_1, \delta_2)}
 < \varepsilon,$

\item $\deg(f,(\delta_1, \delta_2), 0) 
= \deg(g,(\delta_1, \delta_2), 0) = 0 $.
\end{enumerate}
\end{lemma}

By applying the two lemmas above, 
we can establish monotonicity 
properties of functions involving 
variable exponents.

\begin{lemma}\label{lm:strictly_increas}
    Let $g_1, g_2 \in \X.$
Then if we consider the function 
$h : [0,\infty) \to \mathbb{R}$
defined by 
\begin{equation*}
    h(t) = 
\frac{ 
\int_{\Omega} t^{p(x)} g_1(x) \, dx}{ 
\int_{\Omega} t^{q(x)} g_2(x) \, dx}
\end{equation*}
where $1 < q^- \leq q^+ <p^- \leq p^+.$
Then $h$
is strictly increasing.
\end{lemma}

\begin{proof}[\bf Proof.]
Let $h_1:= h|_{[\delta_1, \delta_2]}$
where $0 < \delta_1 < \delta_2.$
Then, by Lemma 
\ref{lem:function_increasing_degree}, 
there exists a function 
$h_4: (\delta_1, \delta_2) \to \mathbb{R}$
given by
\begin{equation*}
    h_4(t) = 
\frac{ \sum\limits_{i=1}^n t^{p_i} 
\int_{\Omega_i} g_1(x) dx}
{ \sum\limits_{j=1}^m t^{q_j} 
\int_{\Omega_j} g_2(x)dx}
\end{equation*}
where $g_1$
and $g_2$
are some given functions and $p_i \in [p^-,p^+]$
and $q_j \in [q^-,q^+]$.
Clearly, $h_4(t)$
is strictly increasing, based on the 
assumptions on $p(x)$
and $q(x)$.

Now, define two auxiliary functions 
$h_2, h_3 : [\delta_1, \delta_2] \times 
[\delta_1, \delta_2] \to \mathbb{R}$
as
\begin{equation*}
    h_2(t_1, t_2) = h_1(t_1) - h_1(t_2)
\end{equation*}
\begin{equation*}
    h_3(t_1, t_2) = h_4(t_1) - h_4(t_2)
\end{equation*}
Let us consider the set 
$$
A = (\delta_1, \delta_2) \times (\delta_1, \delta_2)
 \setminus \{ (t_1, t_2) \mid t_1 = t_2 \}.
 $$
Then, by Lemma \eqref{lem:equal_degree}, 
we have
\begin{equation*}
    \deg(h_2, A, 0) = \deg(h_3, A, 0) = 0
\end{equation*}
the last equality follows from the fact that 
$h_4$
is strictly increasing.
This result implies that $h_1$
is monotonic on the interval 
$(\delta_1, \delta_2).$
Also note that $h$
is increasing near $0$
and
$\lim_{t \to \infty} h(t) = \infty.$
Consequently, $h$
is strictly increasing on $(0, \infty)$.
\end{proof}

\begin{remark}\rm
    Note that if the exponent $q(x)$
is negative in the above lemma, 
the result follows trivially.
Indeed, $h(t)$
will be a product of two strictly 
increasing functions.
\end{remark}

\begin{lemma}\label{lm:norm_bounds}
Let us define the functionals
\begin{align*}
    P(u) & = 
\int_{\Omega} |\grad u|^{p(x)}\,dx 
+ \iint_{\mathbb{R}^{2N}} 
\frac{|u(x) - u(y)|^{p(x,y)}}
{|x - y|^{N + s p(x,y)}}\,dx\,dy,
 \\
    Q(u) & = 
\int_{\Omega} a(x) (u_+)^{1 - \gamma(x)}\,dx,
    \quad 
    R(u) = 
\int_{\Omega} b(x) (u_+)^{r(x)}\,dx.
\end{align*}
Then, the following properties hold:

\begin{enumerate}
    
\item[(i)] If $\|u\|_{\X} > 1,$
then
    \begin{equation*} 
        \|u\|_{\X}^{p^-} \leq P(u) 
        \leq \|u\|_{\X}^{p^+}.
    \end{equation*}
    If $\|u\|_{\X} \leq 1,$
then
    \begin{equation*} 
        \|u\|_{\X}^{p^+} \leq P(u) 
        \leq \|u\|_{\X}^{p^-}.
    \end{equation*}

\item[(ii)] The functional $Q(u)$
satisfies the following inequality:
    \begin{equation*}
        Q(u) \leq C_1 \max 
\left\{ \|u\|_{\X}^{1 - \gamma^+}, 
\|u\|_{\X}^{1 - \gamma^-} 
\right\},
    \end{equation*}
    where $C_1 > 0$
is a constant depending on $N$
and $p(x),$
i.e., $C_1 = C_1(s,N, p(x),a(x),\om)$.

\item[(iii)] The functional $R(u)$
satisfies:
    \begin{equation*} 
        R(u) \leq C_2 \max 
\left\{ \|u\|_{\X}^{r^+}, \|u\|_{\X}^{r^-} 
\right\},
    \end{equation*}
    where $C_2>0$
is a constant depending on $N$
and $p(x),$
i.e., $C_2 = C_2(s,N, p(x),b(x),\om)$.
\end{enumerate}
\end{lemma}

\begin{proof}[\bf Proof.]
$(i)$ Follows directly by 
observing that $P(u) = \rho_{\X}(u)$.

$(ii)$
Using Lemma \ref{lemA1} and 
H\"older inequality with $t(x) = 
\frac{p^{*}(x)}{1 - \gamma(x)}$
and conjugate exponent $t'(x)=
\frac{p^*(x)}{p^*(x)-(1-\ga(x))},$
from $(A_1),$
we obtain
    \begin{align*}
        Q(u) &= 
\int_{\Omega} a(x)(u_+)^{1 - \gamma(x)}\,dx \\
&\leq \|a(x)\|_{L^{t'(x)}(\Omega)} 
\| (u_+)^{1 - \gamma(x)} \|_{L^{t(x)}(\Omega)} \\
&\leq \|a(x)\|_{L^{t'(x)}(\Omega)} 
\max \left\{ \|u\|_{L^{p^{*}(x)}
(\om)}^{1 - \gamma^+}, 
\|u\|_{L^{p^{*}(x)}(\om)}^{1 - \gamma^-}
 \right\} \\
&\leq C_1 \max 
\left\{ \|u\|_{\X}^{1 - \gamma^+}, 
\|u\|_{\X}^{1 - \gamma^-}
 \right\}.
    \end{align*}
    In the final inequality, we have used the Sobolev
     embedding and Lemma \ref{lm:space_embedd}.

$(iii)$
 By applying H\"older inequality, 
the Sobolev embedding theorem, and 
Lemma~\ref{lm:space_embedd}, and proceeding 
similarly to part $(ii),$
we obtain the result.\qedhere
\end{proof}

\begin{lemma} \label{lem:empty_nehari_0}
There exists $\Lambda > 0$
such that for all $\lambda \in (0, \Lambda),$
we have
\begin{equation*}
    \mathcal{N}_\lambda^0 = \{0\}.
\end{equation*}
\end{lemma}

\begin{proof}[\bf Proof.]
We prove this lemma by contradiction.
Let $0 \neq u \in \mathcal{N}_\lambda^0.$
Since $u \in \mathcal{N}_\lambda$
from \eqref{eq:nehari_manifold}, we have $\phi'(1) = 0.$
This implies that
\begin{align*} 
&
\int_\Omega |\grad u|^{p(x)} \, dx + 
\iint_{\mathbb{R}^{2N}} 
\frac{|u(x) - u(y)|^{p(x,y)}}
{|x - y|^{N + s p(x,y)}} \, dx \, dy \\
- 
  \lambda 
& \int_\Omega a(x) (u_+)^{1 - \gamma(x)} \, dx 
- 
\int_\Omega b(x) (u_+)^{r(x)} \, dx = 0.
\end{align*} 
That is,
\begin{equation*} 
P(u) - \lambda Q(u) - R(u) = 0.
\end{equation*}

 Now, suppose $\|u\|_{\X} \leq 1.$
Then using \eqref{eq:second_derivative} and 
Lemma \ref{lm:norm_bounds} yields
\begin{align*}
0 = \phi_u''(1) &\leq (p^+ - 1) P(u) 
+ \lambda \gamma^+ Q(u) - (r^- - 1) R(u) \\
&= (p^+ - 1) P(u) + \lambda \gamma^+ Q(u) - 
(r^- - 1) (P(u) - \lambda Q(u)) \\
&= 
\left (p^+ - r^- 
\right ) P(u) + \lambda (\gamma^+ + r^- - 1) Q(u) \\
&\leq 
\left (p^+ - r^- 
\right ) \|u\|_{\X}^{p^+} + \lambda 
(\gamma^+ + r^- - 1) C_1 \|u\|_{\X}^{1 - \gamma^+}.
\end{align*}
Thus,
\begin{equation} \label{eq:upper_bound_norm}
\|u\|_{\X}^{p^+ -1+ \gamma^+} \leq 
\frac{\lambda C_1 
\left ( \gamma^+ + r^- - 1 
\right )}{r^- - p^+}.
\end{equation}

 Next, we consider another estimate
\begin{align*}
0 = \phi_u''(1) &\geq (p^- - 1) P(u) 
+ \gamma^- (P(u) - R(u)) - (r^+ - 1) R(u) \\
&= (p^- - 1 + \gamma^-) P(u) - (\gamma^- 
+ r^+ - 1) R(u) \\
&\geq (p^- - 1 + \gamma^-) \|u\|_{\X}^{p^+} 
- C_2(\gamma^- + r^+ - 1) \|u\|_{\X}^{r^-}.
\end{align*}
Therefore,
\begin{equation} \label{eq:lower_bound_norm}
\|u\|_{\X}^{r^- - p^+} \geq 
\frac{p^- - 1 + \gamma^-}{C_2(\gamma^- + r^+ - 1)}
\end{equation}
Consequently, from 
\eqref{eq:upper_bound_norm} 
and \eqref{eq:lower_bound_norm}, 
we deduce
\begin{equation*} 
\lambda \geq 
\frac{r^- - p^+}{C_1 (\gamma^+ + r^- - 1)} 
\Big [ 
\frac{p^- - 1 + \gamma^-}{C_2 (\gamma^- + r^+ - 1)} 
\Big ]^{
\frac{(p^+ -1 + \gamma^+)}{r^- - p^+}} := \Lambda_1.
\end{equation*}
 Now, if $\|u\|_{\X} \geq 1,$
then
\begin{align*}
0 = \phi_u''(1) &\leq (p^+ - r^-) \|u\|_{\X}^{p^-} 
+ \lambda (\gamma^+ + r^- - 1) C_1 
\|u\|_{\X}^{1 - \gamma^{-}}, 
\end{align*}
which gives
\begin{equation*}
\|u\|_{\X}^{p^- - 1 + \gamma^-} \leq 
\frac{\lambda C_1 (\gamma^+ + r^- - 1)}{r^- - p^+}.
\end{equation*}
 Similarly, from the lower bound
\begin{align*}
0 = \phi_u''(1) &\geq (p^- - 1 + \gamma^-) 
P(u) - (\gamma^- + r^+ - 1) R(u) \\
&\geq (p^- - 1 + \gamma^-) \|u\|_{\X}^{p^-} 
- (\gamma^- + r^+ - 1) C_2\|u\|_{\X}^{r^+},
\end{align*}
so that
\begin{equation}\label{eq:lower_bound_nehari}
\|u\|_{\X}^{r^+ - p^-} \geq 
\frac{p^- - 1 + \gamma^-}{C_2(\gamma^- + r^+ - 1)}.
\end{equation} 
Consequently, from the second case, we obtain
\begin{equation*}
\lambda \geq 
\frac{r^- - p^+}{C_1(\gamma^+ + r^- - 1)} 
\left[ 
\frac{p^- - 1 + \gamma^-}
{C_2(\gamma^- + r^+ - 1)} \right]^{
\frac{p^- - 1 + \gamma^-}
{r^+ - p^-}} := \Lambda_2.
\end{equation*}
 Choose 
$
\Lambda := \min\{\Lambda_1, \Lambda_2\},
$
then we have our result.
\end{proof}

\begin{lemma}\label{lm:bdd_blw}
The functional $J_\lambda$
is coercive and bounded below on 
$\mathcal{N}_\lambda$.
\end{lemma}

\begin{proof}[\bf Proof.]
Let $u \in \mathcal{N}_\lambda.$
Suppose $\|u\|_{\X} > 1.$
Then
\begin{align*}
J_\lambda(u) &= 
\frac{1}{p^+} P(u) - 
\frac{\lambda}{1 - \gamma^+} Q(u) - 
\frac{1}{r^-} R(u) \\
&\geq 
\Big ( 
\frac{1}{p^+} - 
\frac{1}{r^-} 
\Big ) \|u\|_{\X}^{p^-} - \lambda 
\Big ( 
\frac{1}{1 - \gamma^+} + 
\frac{1}{r^-} 
\Big ) \|u\|_{\X}^{1 - \gamma^-}.
\end{align*}
Thus, $J_\lambda(u) \to \infty$
as $\|u\|_{\X} \to \infty,$
proving coercivity and boundedness.
\end{proof}

\smallskip

\begin{lemma}
Suppose $\lambda \in (0, \Lambda),$
then $\mathcal{N}_\lambda^-$
is closed in the $\X$
topology.
\end{lemma}

\begin{proof}[\bf Proof.]
Let $\{w_k\}$
be a sequence in $\mathcal{N}_\lambda^-$
such that $w_k \to w$
in $\X.$
Then
\[
\lim_{k \to \infty} \rho_{\X}(w_k - w) = 0.
\]
  As $w_k \in \mathcal{N}_\lambda,$
we have
\begin{align*}
\int_\Omega |w_k|^{p(x)} \, dx 
&+ \iint_{\mathbb{R}^{2N}} 
\frac{|w_k(x) - w_k(y)|^{p(x,y)}}
{|x - y|^{N + s p(x,y)}} \, dx \, dy \\
&= \lambda 
\int_\Omega a(x) (w_k(x))_+^{1 - \gamma(x)} \, dx + 
\int_\Omega b(x) (w_k(x))_+^{r(x)} \, dx.
\end{align*}  
By convergence of $w_k \to w$
in $\X,$
we have
\begin{align*}
\int_\Omega |\grad w|^{p(x)} \, dx 
&+ \iint_{\mathbb{R}^{2N}} 
\frac{|w(x) - w(y)|^{p(x,y)}}
{|x - y|^{N + s p(x,y)}} \, dx \, dy \\
&= \lambda 
\int_\Omega a(x) |w(x)|^{1 - \gamma(x)} \, dx + 
\int_\Omega b(x) |w(x)|^{r(x)} \, dx.
\end{align*}
Hence, $w \in \mathcal{N}_\lambda$.

  Similarly, we analyse
\begin{align*}
\phi''_u(1) & =
\int_\Omega p(x) |\grad w|^{r(x)} \, dx+ 
\iint_{\mathbb{R}^{2N}} 
\frac{p(x,y) |w(x)- w(y)|^{p(x,y)}}
{|x - y|^{N + s p(x,y)}} \, dx \, dy\\
& - \la 
\int_{\om}(1-\gamma(x))a(x)(w_+)^{1 - \gamma(x)}- 
\int_\Omega b(x) (w_+)^{r(x)} r(x) \, dx\\
& = \lim_{k \to \infty}
\Big [
\int_\Omega p(x) |w_k(x)|^{p(x)}   dx 
\! + \!
\iint_{\mathbb{R}^{2N}} 
\frac{p(x,y) |w_k(x) - w_k(y)|^{p(x,y)}}
{|x - y|^{N + s p(x,y)}}  dx  dy   
\\
&    - \lambda 
\int_\Omega (1 - \gamma(x)) a(x) 
(w_k(x))_+^{1 - \gamma(x)} \, dx 
- 
\int_\Omega b(x) (w_k(x))_+^{r(x)} r(x) \, dx
\Big ] \\
&\leq 0,
\end{align*}
which implies 
$w \in \mathcal{N}_\lambda^- \cap \mathcal{N}_\lambda^0.$
But following the same calculations as in 
\eqref{eq:lower_bound_nehari}, we have
\[
\|w\|_{\X} = \lim_{k \to \infty} \|w_k\|_{\X} 
\geq A_0 > 0.
\]
Thus, \( w \not\equiv 0 \).
Using the Lemma \ref{lem:empty_nehari_0}, 
we know that $\mathcal{N}_\lambda^0 = \{0\}.$
Consequently, 
\( w \in \mathcal{N}_\lambda^- \).
\end{proof}

We define
$$
\Th_\la=\inf_{u\in \mathcal{N}_\la}J_\la(u);
\quad 
\Th_\la^+=\inf_{u\in \mathcal{N}_\la^+}J_\la(u);
\quad 
\Th_\la^-=\inf_{u\in \mathcal{N}_\la^-}J_\la(u).
$$

\begin{lemma}\label{lm:Nehari_positive}
For $0<\la<\La,$
we have $\Th_\la\le\Th_\la^+<0.$
\end{lemma}

\begin{proof}[\bf Proof.]
    Let $u\in \mathcal{N}_\la^+(\subset\mathcal{N}_\la).$
From \eqref{eq:eneryg_functional} and 
\eqref{eq:nehari_manifold}, we get 
    \begin{align}\label{eq3}
        J_\la(u)
        & = 
\int_\om 
\frac{1}{p(x)}|\grad u|^{p(x)}+\iint_{\mathbb{R}^{2N}} 
\frac{1}{p(x,y)}
\frac{|u(x)-u(y)|^{p(x,y)}}{|x-y|^{N+sp(x,y)}}    \notag \\
& -\la
\int_\om 
\frac{a(x)}{1-\ga(x)}(u_+)^{1-\ga(x)}  
 -
\int_\om 
\frac{b(x)}{r(x)}(u_+)^{r(x)}\nonumber\\
        &\le 
\frac{1}{p^-}P(u)-
\frac{\la}{1-\ga^-}Q(u)-
\frac{1}{r^+}R(u)\nonumber\\
        & \le 
\Big (
\frac{1}{p^-}-
\frac{1}{1-\ga^-} 
\Big )  P(u)+ 
\Big (
\frac{1}{1-\ga^-}-
\frac{1}{r^+} \Big )
 R(u).
    \end{align}
    Since $u\in \mathcal{N}_\la^+,$
from \eqref{eq:second_derivative}, we also have 
     \begin{align*}
0<\varphi''_u(1)&\le p^+P(u)-\la(1-\ga^+)Q(u)-r^-R(u)\\
        &\le (p^+-(1-\ga^+))P(u)-\la(r^--(1-\ga^+))R(u),
    \end{align*}
      that is,
       \begin{equation}\label{eq4}
R(u)\le 
\frac{p^+-(1-\ga^+)}{r^--(1-\ga^+)}P(u).
    \end{equation}
    From \eqref{eq3} and \eqref{eq4}, we can write
        \begin{align*}
        J_\la(u)
&\le \Big (
\frac{1}{p^-}-
\frac{1}{1-\ga^-} \Big ) P(u)+ 
\Big (
\frac{1}{1-\ga^-}-
\frac{1}{r^+} \Big ) R(u)\\
        & \le \Big (
\frac{1}{p^-}-
\frac{1}{1-\ga^-} \Big ) P(u)+ \Big (
\frac{1}{1-\ga^-}-
\frac{1}{r^+} \Big ) \Big (
\frac{p^+-(1-\ga^+)}{r^--(1-\ga^+)}\Big ) P(u)\\
        &\le 
\Big [ \Big (
\frac{1}{p^-}-
\frac{1}{1-\ga^-} \Big ) + \Big (
\frac{1}{1-\ga^-}-
\frac{1}{r^+} \Big ) \Big (
\frac{p^+-(1-\ga^+)}{r^--(1-\ga^+)}\Big ) \Big ] P(u).
    \end{align*}
    From assumption \eqref{eq5}, we conclude
    $$\inf_{u\in \mathcal{N}_\la^+}J_\la(u)< 0.$$
   Using the definition of $\Th_\la^+$
and $\Th_\la,$
we get the result.
\end{proof}

\begin{lemma}\label{lm:Nehari_negative}
For $0<\la<\frac{1-\ga^+}{p^+}\La,$
there exists a positive constant $K$
such that $\Th_\la^-\ge K$.
\end{lemma}

\begin{proof}[\bf Proof.]
Let 
$u\in \mathcal{N}_\la^-(\subset\mathcal{N}_\la).$
From \eqref{eq:lower_bound_norm} and 
\eqref{eq:lower_bound_nehari}, we have
        \begin{equation}\label{eq20}
            \begin{cases}
            \|u\|_{\X} \geq 
 \left[
\frac{p^- - 1 + \gamma^-}{C_2(\gamma^- + r^+ - 1)}\right]^
\frac{1}{{r^- - p^+}} &\text{ if }\|u\|_{\X}< 1,\\[5pt]
\|u\|_{\X} \geq \left[
\frac{p^- - 1 + \gamma^-}{C_2(\gamma^- + r^+ - 1)}\right]^
\frac{1}{{r^+ - p^-}} &\text{ if }\|u\|_{\X}> 1.
            \end{cases}
        \end{equation}
        First, we consider the case $\|u\|_{\X}<1.$
From \eqref{eq:eneryg_functional} and 
\eqref{eq:nehari_manifold}, using Lemma 
\ref{lm:norm_bounds} and equation 
\eqref{eq20}, we obtain
        \begin{align*}
        J_\la(u) & \ge 
\frac{1}{p^+}P(u)-
\frac{\la}{1-\ga^+}Q(u)-
\frac{1}{r^-}R(u)\\
         & \ge 
\frac{1}{p^+}P(u)-
\frac{\la}{1-\ga^+}Q(u)-
\frac{1}{r^-}(P(u)-\la Q(u))\\
         & \ge  
\Big (
\frac{1}{p^+}-
\frac{1}{r^-} 
\Big ) P(u)-\la   
\Big (
\frac{1}{1-\ga^+}-
\frac{1}{r^-}  
\Big ) Q(u)\\
         & \ge    
\Big (
\frac{1}{p^+}-
\frac{1}{r^-}  
\Big ) \|u\|_{\X}^{p^+}-\la   
\Big (
\frac{1}{1-\ga^+}-
\frac{1}{r^-}  
\Big ) C_1\|u\|_{\X}^{1-\ga^+}\\
         & \ge \|u\|_{\X}^{1-\ga^+}\lsb  
\Big (
\frac{1}{p^+}-
\frac{1}{r^-}  
\Big ) \|u\|_{\X}^{p^+-(1-\ga^+)}-\la C_1   
\Big (
\frac{1}{1-\ga^+}-
\frac{1}{r^-}  
\Big )\rsb\\
         & \ge
\Big [
\frac{p^- - 1 + \gamma^-}{C_2(\gamma^- + r^+ - 1)}
\Big ]^
\frac{1-\ga^+}{{r^- - p^+}}
\Big [  
\Big (
\frac{1}{p^+}-
\frac{1}{r^-}  
\Big ) 
\Big  (
\frac{p^- - 1 + \gamma^-}{C_2(\gamma^- + r^+ - 1)}
\Big  )^
\frac{p^+-(1-\ga^+)}{{r^- - p^+}}  \\
        &
        -\la C_1   
\Big (
\frac{1}{1-\ga^+}-
\frac{1}{r^-}  
\Big )\Big ]:=k_1
    \end{align*}
    We can observe that $k_1>0$
if 
    $$
    \Big [ \Big (
\frac{1}{p^+}-
\frac{1}{r^-} \Big ) 
\Big  (
\frac{p^- - 1 + \gamma^-}{C_2(\gamma^- + r^+ - 1)}
\Big  )^
\frac{p^+-(1-\ga^+)}{{r^- - p^+}}-\la C_1 \Big (
\frac{1}{1-\ga^+}-
\frac{1}{r^-} \Big )\Big ]>0$$
    which implies that 
     $$
     \la<
\frac{1-\ga^+}{p^+}
\frac{1}{C_1}\Big (
\frac{r^--p^+}{r^--(1-\ga^+)}\Big ) 
\Big (
\frac{p^- - 1 + \gamma^-}{C_2(\gamma^- + r^+ - 1)}
\Big )^
\frac{p^+-(1-\ga^+)}{{r^- - p^+}}=
\frac{1-\ga^+}{p^+}\La_1.$$
     Next, we consider the case 
$\|u\|_{\X}>1.$
As in the previous case, by applying 
\eqref{eq:eneryg_functional}, 
\eqref{eq:nehari_manifold}, 
Lemma~\ref{lm:norm_bounds}, and equation 
\eqref{eq20}, we obtain
     \begin{align*}
        J_\la(u) & \ge 
\frac{1}{p^+}P(u)-
\frac{\la}{1-\ga^+}Q(u)-
\frac{1}{r^-}R(u)\\
         & \ge  \Big (
\frac{1}{p^+}-
\frac{1}{r^-} \Big ) P(u)-\la \Big (
\frac{1}{1-\ga^+}-
\frac{1}{r^-} \Big ) Q(u)\\
         & \ge  \Big (
\frac{1}{p^+}-
\frac{1}{r^-} \Big ) \|u\|_{\X}^{p^-}-\la \Big (
\frac{1}{1-\ga^+}-
\frac{1}{r^-} \Big ) C_1\|u\|_{\X}^{1-\ga^-}\\
         & \ge \|u\|_{\X}^{1-\ga^-}\lsb\Big (
\frac{1}{p^+}-
\frac{1}{r^-} \Big ) \|u\|_{\X}^{p^+-(1-\ga^-)}-\la C_1 \Big (
\frac{1}{1-\ga^+}-
\frac{1}{r^-} \Big )\rsb\\
         & \ge  \Big [
\frac{p^- - 1 + \gamma^-}{C_2(\gamma^- + r^+ - 1)}\Big ]^
\frac{1-\ga^-}{{r^+ - p^-}}\Big [\Big (
\frac{1}{p^+}-
\frac{1}{r^-} \Big ) 
\Big  (
\frac{p^- - 1 + \gamma^-}{C_2(\gamma^- + r^+ - 1)}
\Big  )^
\frac{p^--(1-\ga^-)}{{r^+ - p^-}} \\
        &  -\la C_1 \Big (
\frac{1}{1-\ga^+}-
\frac{1}{r^-} \Big )\Big ]:=k_2
    \end{align*}
    We can observe that $k_2>0$
if 
    $$\Big [ \Big (
\frac{1}{p^+}-
\frac{1}{r^-} \Big ) 
\Big  (
\frac{p^- - 1 + \gamma^-}{C_2(\gamma^- + r^+ - 1)}
\Big  )^
\frac{p^--(1-\ga^-)}{{r^+ - p^-}}-\la C_1 \Big (
\frac{1}{1-\ga^+}-
\frac{1}{r^-} \Big )\Big ]>0$$
    which implies that 
     $$\la<
\frac{1-\ga^+}{p^+}
\frac{1}{C_1}\Big (
\frac{r^--p^+}{r^--(1-\ga^+)}\Big ) 
\Big  (
\frac{p^- - 1 + \gamma^-}{C_2(\gamma^- + r^+ - 1)}
\Big  )^
\frac{p^--(1-\ga^-)}{{r^+ - p^-}}=
\frac{1-\ga^+}{p^+}\La_2.$$
     Choosing $K=\min\{k_1,k_2\}>0,$
we get the result.
\end{proof}
Now, with all the required lemmas in hand, 
we can prove the following results.

\begin{lemma}\label{lem:properties_of_fi}
Let $f_1, f_2, f_3$
be as defined in \eqref{functions}.
Then, the following properties hold:
\begin{enumerate}
    
\item[$(1)$] $\displaystyle \lim_{t \to 0^+} 
\frac{f_3(t)}{f_1(t)} = 0$.
    
\item[$(2)$] $\displaystyle \lim_{t \to 0} f_2(t) = \infty$.
    
\item[$(3)$] $\displaystyle \lim_{t \to 0} 
\frac{(f_1 - f_3)(t)}{f_2(t)} = 0$.
    
\item[$(4)$] $f_1 - f_3$
has a unique point of maxima, namely $t_{\max}$.
    
\item[$(5)$] There exists $\tilde{t} \in (0, t_{\max})$
such that $ 
\frac{f_1 - f_3}{f_2}$
is strictly increasing on $(0, \tilde{t})$.
\end{enumerate}
\end{lemma}

\begin{proof}[\bf Proof.]
$(1),$ $(2)$,
and $(3)$
follow directly.

$(4)$
We note that 
$
f_1'(t) - f_3'(t) > 0$
 for any sufficiently small $ t.$
Also,
\[
\lim_{t \to \infty} \left[ f_1'(t) - f_3'(t) \right]
 = -\infty.
\]
So to prove our assertion, it is enough to prove
$f_1'(t_0) - f_3'(t_0) = 0$
for some unique $t_0 > 0.$
This is the same as
\[
\frac{f_1'(t_0)}{f_3'(t_0)} = 1 \quad 
\text{at some unique point } t_0 > 0.
\]
This holds by Lemma 
\ref{lem:function_increasing_degree}.

(5) Clearly, $
\frac{f_1 - f_3}{f_2}$
is increasing in a neighbourhood of 0.

Let 
\[
\tilde{t} = \sup 
\Big \{ t > 0 \,\Big |\,
\Big  (
\frac{f_1 - f_3}{f_2} 
\Big  ): [0, t) \to \mathbb{R} 
\text{ is increasing} 
\Big \}.
\]
Clearly $\tilde{t} > 0.$
Now, if $\tilde{t} > t_{\max},$
then
\begin{equation*}
f_1 - f_3 = f_2 
\Big  ( 
\frac{f_1 - f_3}{f_2} 
\Big  )
\end{equation*}
is strictly increasing on $[0, \tilde{t}),$
which leads to a contradiction.
\end{proof}

With the previous lemma in hand, we now 
provide a complete characterization of the 
geometry of the fibering maps associated 
with the problem.

\begin{lemma}\label{lm:fiber_map_analysis}
Let $u \in \X \setminus \{0\}$
be a fixed function.
Then
\begin{itemize}
    
\item[(i)] Assume that $f_3=0.$
Then there exists a unique $t^+ > 0$
such that 
    \begin{equation*}
t^+ u \in \mathcal{N}_\lambda^+ \text{ and } 
J_\la(t^+u)=\ds\inf_{t\ge 0}J_\la(tu).
\end{equation*}
    
\item[(ii)] Assume that $f_3 > 0.$
Then there exists $t^*>0$
and unique numbers $t^+,t^->0$
with $t^+< t^-$
such that
    \begin{equation*}
t^+ u \in \mathcal{N}_\lambda^+ \quad \text{and} 
\quad t^- u \in \mathcal{N}_\lambda^-.
    \end{equation*}
    Moreover,
\begin{equation*}
J_\la(t^+u)=\ds\inf_{0\le t\le t^*}J_\la(tu) 
\quad\text{ and }\quad
J_\la(t^-u)=\ds\sup_{t\ge t^+}J_\la(tu).
			\end{equation*}
     
\end{itemize}
\end{lemma}

\begin{proof}[\bf Proof.]
\textbf{$(i)$} We begin by noting the following limits
\begin{equation*}
\lim_{t \to 0} 
\frac{f_1(t)}{f_2(t)} = 0 \quad\text{and} 
\quad\lim_{t \to \infty} 
\frac{f_1(t)}{f_2(t)} = \infty.
\end{equation*}
Moreover, from Lemma~\ref{lm:strictly_increas}, 
we know that $\frac{f_1}{f_2}$
is strictly increasing.

 Thus, there exists a unique point $t^+$
such that
\begin{equation*}
\Big  ( \frac{f_1 }{f_2} 
\Big  )(t^+) = \lambda. 
\end{equation*}
From \eqref{fibermap_derivative_f_i}, 
we have $\phi'_u(t^+)=0,\,\phi'_u(t^+)<0$
for all $t<t^+$
and $\phi'_u(t^+)>0$
for all $t>t^+.$
Thus, $\phi_u(t)$
attains local minimum at $t=t^+$
and $\phi''_u(t^+)>0.$
Therefore,
\begin{equation*}
t^+ u \in \mathcal{N}_\lambda^+ 
\text{ and } J_\la(t^+u)=\ds\inf_{t\ge 0}J_\la(tu).
\end{equation*}

\textbf{$(ii)$} Since $  \lim_{t \to 0^+} 
\frac{f_3(t)}{f_1(t)} = 0,$
we observe that $(f_1 - f_3)(t) > 0$
for sufficiently small $t$.

 Now, using part $(2)$
of Lemma~\ref{lem:properties_of_fi}, 
i.e.,  $\lim_{t \to 0} f_2(t) = \infty$
and from part $(4)$
of Lemma~\ref{lem:properties_of_fi}, 
we conclude that there exists 
$t_1 \in (0, \tilde{t})$
 such that
\begin{equation*}
\lambda f_2(t_1) > (f_1 - f_3)(t_1). 
\end{equation*}
 Now, for any fixed sufficiently small $\lambda > 0,$
we have
\begin{equation*}
\lambda f_2(t_2) < f_1(t_2) - f_3(t_2) 
\quad \text{for some } t_2 \in (0, \tilde{t}). 
\end{equation*}
 Then, by the Mean Value Theorem, 
 there exists $t^+ \in (0, \tilde{t})$
such that
\begin{equation*}
\lambda f_2(t^+) = f_1(t^+) - f_3(t^+). 
\end{equation*} 
Since $f_1 - f_3$
is increasing on $(0, \tilde{t})$
and $f_2$
is strictly decreasing on $(0, \tilde{t}),$
it follows that
\begin{equation*}
f_1 - f_3 - \lambda f_2 
\text{ is increasing on } (0, \tilde{t}). 
\end{equation*}
This implies that $t^+$
is unique.

Now, using part $(5)$
of Lemma~\ref{lem:properties_of_fi}, 
we know that $\frac{f_1 - f_3}{f_2}$
is strictly increasing on $(t^+, \tilde{t}).$
It follows that
\begin{equation*}
\lambda = 
\frac{(f_1 - f_3)(t^+)}{f_2(t^+)} < 
\frac{(f_1 - f_3)(t)}{f_2(t)} \quad 
\text{ for all } t \in (t^+, \tilde{t}). 
\end{equation*}
 Consequently, we obtain
\begin{equation*}
\lambda f_2(t) < (f_1 - f_3)(t) \quad 
\text{ for all } t \in (t^+, \tilde{t}). 
\end{equation*}
 Fix $\lambda^*$
sufficiently small such that for $\la\in(0,\la^*),$
we get
\begin{equation*}
\lambda f_2(t) < (f_1 - f_3)(t) \quad 
\text{ for all } t \in (t^+, t_{\max}). 
\end{equation*}
 Also, observe that
\begin{equation*}
f_1 - f_3 - \lambda f_2 \to -\infty \quad
 \text{as } t \to \infty. 
\end{equation*}
 Therefore, there exists $t_3 > t_{\max}$
such that
\begin{equation*}
(f_1 - f_3)(t_3) < \lambda f_2(t_3).
\end{equation*}
 Thus, for all $\lambda \in (0, \lambda^*),$
there exists $t^- > t_{\max}$
such that
\begin{equation*}
\lambda f_2(t^-) = (f_1 - f_3)(t^-).
\end{equation*}
From part $(5)$
of Lemma \ref{lm:fiber_map_analysis}, we have $f_1-f_3$
is strictly decreasing on $(t_{\max}, \infty).$
By fixing $\lambda_2$
sufficiently small, we get $f_1 - f_3 - \lambda f_2$
is strictly decreasing on $(t_{\max}, \infty).$
Therefore, $t^-$
is unique.
Hence, $t^+ < t^-$
are two distinct critical points of $\phi_u(t)$.

Since $\phi_u(0)=0$
and $\phi_u(t)<0$
for $t$
small enough.
We can choose $\la$
sufficiently small specifically, 
$\la=\min\{\la^*,\la_2\}$
such that
$$
\phi_u'(t) < 0 \quad \text{for all } 
t \in (0, t^+),   \text{ and} 
$$
$$
 \phi_u'(t) > 0 \quad \text{for all }
 t \in (t^+, t_{\max})  \ \ \
 $$
  and  
 $ 
\phi_u'(t^+) = 0.
$
Therefore, $\phi_u$
attains a minimum at $t^+$
and $\phi_u''(t^+) > 0.$
Thus,
$
t^+ u \in \mathcal{N}_\lambda^+.
$

Similarly, since $\phi_u'(t) > 0$
for all $t \in (t_{\max}, t^-)$
and $\phi_u'(t) < 0$
for all $t \in (t^-, \infty)$
and $\phi_u'(t^-) = 0.$
Therefore, $\phi_u$
attains a maximum at $t^-$
and $\phi_u''(t^-) < 0.$
Thus,
$
t^- u \in \mathcal{N}_\lambda^-.
$

From Lemma \ref{lm:Nehari_positive} and 
\ref{lm:Nehari_negative}, 
we already know that $\phi_u(t^+)<0$
and $\phi_u(t^-)>0.$
Since, $\phi_u(t)$
is strictly increasing on $[t^+,t^-]$
and strictly decreasing for $t>t^-,$
also $\phi_u(t)\to-\infty$
as $t\to\infty.$
Hence, we can conclude that there exists 
a unique $t^*\in (t^+,t^-)$
such that $\phi_u(t^*)=0$
and 
\begin{equation*}
J_\la(t^+u)=\ds\inf_{0\le t\le t^*}J_\la(tu)
 \quad\text{ and }\quad 
J_\la(t^-u)=\ds\sup_{t\ge t^+}J_\la(tu).
\qedhere
\end{equation*}
\end{proof}

\begin{lemma}\label{lm9}
    Let $u\in \mathcal{N}_\la^\pm,$
then there exists $\varepsilon>0$
and a continuous function 
$f:B_\varepsilon(0)\ra \R^+$
such that
$$
    f(u)>0, f(0)=1 \text{ and } 
    f(v)(u+v)\in \mathcal{N}_\la^\pm 
    \text{ for all } v\in B_\varepsilon(0),
$$
    where $B_\varepsilon(0):=
    \{v\in \X:\|v\|_{\X}<\varepsilon\}.$
\end{lemma}

\begin{proof}[\bf Proof.]
    We consider the case $u\in \mathcal{N}_\la^+;$
the case $u\in \mathcal{N}_\la^-$
follows by a similar argument.
Let $u\in \mathcal{N}_\la^+,$
we define the function $F:\X\times\R^+\ra\R$
by
    \begin{align*}
        F(v,t) & =
\int_\om  t^{p(x)-1}|\grad (u+v)|^{p(x)}\\
&
+\iint_{\mathbb{R}^{2N}} t^{p(x,y)-1}
\frac{|(u+v)(x)-(u+v)(y)|^{p(x,y)}}{|x-y|^{N+sp(x,y)}}\\
        &-\la
\int_\om t^{-\ga(x)}a(x)(u+v)_+^{1-\ga(x)}
        -
\int_\om t^{r(x)-1}b(x)(u+v)_+^{r(x)}.
    \end{align*}
    Since $u\in \mathcal{N}_\la^+\subset\Nh,$
we get
    \begin{align*}
        F(0,1) & =
\int_\om|\grad u|^{p(x)}+\iint_{\mathbb{R}^{2N}}
\frac{|u(x)-u(y)|^{p(x,y)}}{|x-y|^{N+sp(x,y)}} \\
&
-\la
\int_\om a(x)u_+^{1-\ga(x)}
        -
\int_\om b(x)u_+^{r(x)}=0
    \end{align*}
    and 
    \begin{align*}
        \frac{\pa}{\pa t}F(0,1)=
\int_\om &(p(x)-1)|\grad u|^{p(x)}+
\iint_{\mathbb{R}^{2N}} (p(x,y)-1))
\frac{|u(x)-u(y)|^{p(x,y)}}{|x-y|^{N+sp(x,y)}}\\
        &+\la
\int_\om\ga(x) a(x)u_+^{1-\ga(x)}
        -
\int_\om (r(x)-1)b(x)u_+^{r(x)}>0.
    \end{align*}
With the application of implicit function 
theorem at $(0,1),$
there exists $\varepsilon>0$
and a continuous function 
$f:B_\varepsilon(0)\ra \R^+$
such that $t=f(v)$
is a unique solution of $F(v,t)=0.$
From $F(0,1)=0,$
we get $f(0)=1$
and from $F(v,f(v))=0,$
we have
  \begin{align*}
        0 & =
\int_\om 
 f(v)^{p(x)-1}|\grad (u+v)|^{p(x)} \\
 &
+\iint_{\mathbb{R}^{2N}} f(v)^{p(x,y)-1}
\frac{|(u+v)(x)-(u+v)(y)|^{p(x,y)}}{|x-y|^{N+sp(x,y)}}\\
&-\la
\int_\om f(v)^{-\ga(x)}a(x)(u+v)_+^{1-\ga(x)}
        -
\int_\om f(v)^{r(x)-1}b(x)(u+v)_+^{r(x)}\\
& =
\frac{P(f(v)(u+v))-\la Q(f(v)(u+v))-R(f(v)(u+v))}{f(v)}
    \end{align*}
    which implies that 
    $$f(v)(u+v)\in \Nh 
    \text{ for all }v\in \X, \|v\|_{\X}<\ve.$$
Similarly, using 
$\frac{\pa}{\pa t}F(0,1)>0$
and taking $\ve$
even smaller if necessary, we get
$$
f(v)(u+v)\in \Nh^+ \text{ for all }v\in \X, \|v\|_{\X}<\ve.
$$
    This completes the proof of the Lemma \ref{lm9}.
\end{proof}

\section{Existence of solutions in 
\texorpdfstring{$\Nh^{\pm}$}{Nehari Manifold}}\label{sec3}

\begin{lemma}\label{lm:min_pos}
    For $0<\la<\La,$
then $J_\la$
has a minimizer $u_0\in \Nh^+$
and satisfies the following
$$
    J_\la(u_0)=
 \inf_{u\in \mathcal{N}_\la^+}J_\la(u)=\Th_\la^+<0.
$$
\end{lemma}

\begin{proof}[\bf Proof.]
    From Lemma \ref{lm:bdd_blw}, we know that $J_\la$
is bounded below on $\Nh$
and hence on $\Nh^+$
as well.
Therefore, there exists a minimizing sequence $\{u_n\}\in\Nh^+,$
such that
    $$\lim_{n\to\infty} J_\la(u_n)=\inf_{u\in\Nh^+}J_\la(u).$$
    Since $J_\la$
is coercive on $\X,$
it follows that the sequence $\{u_n\}$
is bounded on $\X.$
Thus, there exists a subsequence of $\{u_n\}$
(still denoted by $\{u_n\}$), such that $u_n\weak u_0$
weakly in $\X.$
Using Lemma \ref{lm:space_embedd}, we get
\begin{equation*}
   \left\{
	\begin{aligned}
		&u_n\ra u_0 \text{ strongly in }L^{q(x)}(\om),\\
        &u_n\ra u_0 \text{ a.e. in } \om.
	\end{aligned}
    \right.
\end{equation*}
Now, by applying Lemma \ref{lem:modular_cgs_lp} 
together with the Lebesgue dominated convergence
 theorem, we have
\begin{equation}\label{eq7}
\left.
	\begin{aligned}
    \lim_{n\to\infty}
\int_\om 
\frac{a(x)}{1-\ga(x)}(u_n)_+^{1-\ga(x)}&=
\int_\om 
\frac{a(x)}{1-\ga(x)}(u_0)_+^{1-\ga(x)};\\
    \lim_{n\to\infty}
\int_\om a(x)(u_n)_+^{1-\ga(x)}&=
\int_\om a(x)(u_0)_+^{1-\ga(x)}
\end{aligned}
  \right.
  \end{equation}
and 
\begin{equation}\label{eq22}
\begin{aligned}
    \lim_{n\to\infty}
\int_\om 
\frac{b(x)}{r(x)}(u_n)_+^{r(x)}&=
\int_\om 
\frac{b(x)}{r(x)}(u_0)_+^{r(x)};\\
    \lim_{n\to\infty}
\int_\om b(x)(u_n)_+^{r(x)}&=
\int_\om b(x)(u_0)_+^{r(x)}.
\end{aligned}
  \end{equation}
  We claim that $u_0\neq0.$
First, we prove that for $u_0\in \Nh^+$
it holds that $Q(u_0)>0.$
Since $u_0\in \Nh^+,$
we have
   \begin{align*}
0<\varphi''_{u_0}(1)&\le p^+P(u_0)-\la(1-\ga^+)Q(u_0)-r^-R(u_0)\\
        &\le -(r^--p^+)P(u_0)+\la(r^--(1-\ga^+))Q(u_0)
    \end{align*}
    which implies that
    $$Q(u_0)> 
\frac{(r^--p^+)}{\la(r^--(1-\ga^+))}P(u_0)\geq 0.$$
    Thus, $Q(u_0)>0$
and hence $u_0\ne0.$
If, on the contrary, we assume that $u_0=0.$
Then, from \eqref{eq7}, we get
    \begin{equation*}
    \lim_{n\to\infty}
\int_\om a(x)(u_n)_+^{1-\ga(x)}=
\int_\om a(x)(u_0)_+^{1-\ga(x)}=0,
  \end{equation*}
  that is $Q(u_0)=0,$
which is a contradiction.
Hence, we conclude that $u_0\in\X\setminus\{0\}.$

    It only remains to prove that $u_n\ra u_0$
strongly in $\X.$
Suppose not.
Then $u_n\nrightarrow u_0$
in $\X$
as $n\to\infty.$
Therefore, by applying Lemma 
\ref{lem:modular_cgs_prob} and
\cite[Lemma 2.4.17]{Variable_book}, 
up to a subsequence it follows that
        \begin{align}\label{eq8}
        &
\int_\om 
\frac{1}{p(x)}|\grad u_0|^{p(x)} +\iint_{\mathbb{R}^{2N}} 
\frac{1}{p(x,y)}
\frac{|u_0(x)-u_0(y)|^{p(x,y)}}{|x-y|^{N+sp(x,y)}} \nonumber\\
        &<\liminf_{n\to\infty}
\int_\om 
\frac{1}{p(x)}|\grad u_n|^{p(x)}+\iint_{\mathbb{R}^{2N}} 
\frac{1}{p(x,y)}
\frac{|u_n(x)-u_n(y)|^{p(x,y)}}{|x-y|^{N+sp(x,y)}}.
    \end{align}
    Using \eqref{eq:eneryg_functional}, 
    \eqref{eq7}, \eqref{eq22}, and \eqref{eq8}, we get
\begin{align}\label{eq9}
\lim_{n\to \infty} J_\la(u_n)
& =
\liminf_{n\to \infty}
\Big [
\int_\om 
\frac{1}{p(x)} |\grad u_n|^{p(x)}
\! + \! 
\iint_{\mathbb{R}^{2N}} 
\frac{1}{p(x,y)}
\frac{|u_n(x)-u_n(y)|^{p(x,y)}}
{|x-y|^{N+sp(x,y)}} \nonumber\\
        &  -\la
\int_\om 
\frac{a(x)}{1-\ga(x)}(u_n)_+^{1-\ga(x)}
        -
\int_\om 
\frac{b(x)}{r(x)}(u_n)_+^{r(x)}\Big ]\nonumber\\
        \geq \liminf_{n\to \infty}
        & \Big [
\int_\om 
\frac{1}{p(x)}|\grad u_n|^{p(x)}
+\iint_{\mathbb{R}^{2N}} 
\frac{1}{p(x,y)}
\frac{|u_n(x)-u_n(y)|^{p(x,y)}}
{|x-y|^{N+sp(x,y)}}\Big ]\nonumber\\
               & -\lim_{n\to\infty}\Big [\la
\int_\om 
\frac{a(x)}{1-\ga(x)}(u_n)_+^{1-\ga(x)}\Big ]
-\lim_{n\to\infty}\Big [
\int_\om 
\frac{b(x)}{r(x)}(u_n)_+^{r(x)}\Big ]\nonumber\\
        >
\int_\om 
\frac{1}{p(x)}&|\grad u_0|^{p(x)}
+\iint_{\mathbb{R}^{2N}}
\frac{1}{p(x,y)}
\frac{|u_0(x)-u_0(y)|^{p(x,y)}}
{|x-y|^{N+sp(x,y)}}\nonumber\\
                &-\la
\int_\om 
\frac{a(x)}{1-\ga(x)}(u_0)_+^{1-\ga(x)}    
\int_\om 
\frac{b(x)}{r(x)}(u_0)_+^{r(x)}\nonumber\\
        =J_\la(u_0).
    \end{align}
By applying Lemma 
\ref{lm:fiber_map_analysis} for 
$u_0\in \X\setminus\{0\},$
we find that in both the cases whether 
$R(u_0)=0$
or $R(u_0)>0$
there exists $0<t_0^+\in \real^+$
such that $t_0^+u_0\in\Nh^+.$
Since we have assumed that 
$u_n\nrightarrow u_0$
in $\X$
as $n\to\infty,$
it follows that
\begin{equation}\label{eq10}
    \rh(t_0^+u_0)< \liminf_{n\to\infty} 
    \rh(t_0^+u_n).
\end{equation}
Also, using Lemma \ref{lem:modular_cgs_lp} 
and Lebesgue dominated convergence theorem, 
we get
\begin{equation}\label{eq11}
    \lim_{n\to\infty}
\int_\om a(x)(t_0^+u_n)_+^{1-\ga(x)}=
\int_\om a(x)(t_0^+u_0)_+^{1-\ga(x)}
  \end{equation}
and 
\begin{equation}\label{eq12}
    \lim_{n\to\infty}
\int_\om b(x)(t_0^+u_n)_+^{r(x)}=
\int_\om b(x)(t_0^+u_0)_+^{r(x)}.
  \end{equation}
From equations 
\eqref{eq:fibering_map_first_derivative}, 
\eqref{eq10}, \eqref{eq11}, and 
\eqref{eq12}, 
we conclude that
    \begin{align} \label{eq13}
\lmt & \varphi'_{u_n}(t_0^+)
 =\liminf_{t\to\infty}\Big [
\int_\om t_0^{p(x)-1}|\grad 
(u_n)|^{p(x)} \notag \\
        &\quad
+\iint_{\mathbb{R}^{2N}}
 t_0^{p(x,y)-1}
\frac{|(u_n)(x)-(u_n)(y)|^{p(x,y)}}
{|x-y|^{N+sp(x,y)}}\Big .\nonumber\\
        &\quad -\la
\int_\om t_0^{-\ga(x)}a(x)(u_n)_+^{1-\ga(x)}
        -
\int_\om t_0^{r(x)-1}b(x)(u_n)_+^{r(x)}
\Big ]\nonumber\\
       &\ge
\frac{1}{t_0^+}\Big [\liminf_{n\to\infty}
\rh(t_0^+u_n)-\la\lmt
\int_\om t_0^{1-\ga(x)}a(x)(u_n)_+^{1-\ga(x)}
 \nonumber\\
        &\quad -\lmt
\int_\om t_0^{r(x)}b(x)(u_n)_+^{r(x)}
\Big ]\nonumber\\
        &>
\frac{1}{t_0^+}\Big [\rh(t_0^+u_0)-\la
\int_\om t_0^{1-\ga(x)}a(x)(u_0)_+^{1-\ga(x)}-
\int_\om t_0^{r(x)}b(x)(u_0)_+^{r(x)}
\Big ]\nonumber\\
        &=\varphi'_{u_0}(t_0^+)=0.
    \end{align}
Hence, for sufficiently large $n,$
we have $\ds \lmt \varphi'_{u_n}(t_0^+)>0.$
Since $u_n\in \Nh^+$
for all $n\in\mathbb{N},$
it holds that $\varphi_{u_n}'(1)=0$
and $\varphi_{u_n}''(1)>0.$
By Lemma \ref{lm:fiber_map_analysis}, 
this implies that $\varphi_{u_n}'(t)<0$
for all $t\in (0,1).$
Therefore, from \eqref{eq13}, 
we conclude that $t_0^+>1.$
Since $t_0^+u_0\in\Nh^+$
and by Lemma \ref{lm:fiber_map_analysis}, 
$\varphi_{u_0}(t)$
is decreasing on $(0,t_0^+).$
Thus, from \eqref{eq9}, we obtain
$$
J_\la(t_0^+u_0)\le J_\la(u_0)<\lmt 
J_\la(u_n)=\inf_{u\in\Nh^+}J_\la(u),
$$
which contradicts the fact that 
$t_0^+u_0\in\Nh^+.$
Hence, $u_n\rightarrow u_0$
strongly in $\X$
as $n\to\infty$
and thus $u_0\in\Nh.$
But by Lemma \ref{lem:empty_nehari_0}, 
$\Nh^0=\{0\}$
and since by Lemma 
\ref{lm:Nehari_positive} 
$J_\la(u_0)=\ds \lmt J_\la(u_n)<0,$
we conclude that $u_0\in\Nh^+.$
\end{proof}

\begin{lemma}\label{lm:min_neg}
    For $0<\la<\La,$
then $J_\la$
has a minimizer $z_0\in \Nh^-$
and satisfies the following
$$
    J_\la(z_0)=
    \inf_{u\in \mathcal{N}_\la^-}J_\la(u)=\Th_\la^->0.
$$
\end{lemma}

\begin{proof}[\bf Proof.]
      From Lemma \ref{lm:bdd_blw}, 
      it follows that $J_\la$
is bounded below on $\Nh^-.$
Thus, there exists a minimizing sequence 
$\{z_n\}\in\Nh^-,$
such that
$$
    \lim_{n\to\infty} J_\la(z_n)=\inf_{z\in\Nh^-}J_\la(z).
$$
    Since $J_\la$
is coercive, the sequence $\{z_n\}$
is bounded on $\X.$
Hence, up to a subsequence 
(still denoted by $\{z_n\}$), 
we have $z_n\weak z_0$
weakly in $\X.$
Applying the Lemma \ref{lm:space_embedd}, 
we obtain
    \begin{equation*}
   \left\{
	\begin{aligned}
&z_n\ra z_0 \text{ strongly in }L^{q(x)}(\om),\\
&z_n\ra z_0 \text{ a.e. in } \om.
	\end{aligned}
    \right.
\end{equation*}
Now, by applying Lemma \ref{lem:modular_cgs_lp} 
together with the Lebesgue dominated 
convergence theorem, we have
\begin{equation}\label{eq29}
	\begin{aligned}
    \lim_{n\to\infty}
\int_\om 
\frac{a(x)}{1-\ga(x)}(z_n)_+^{1-\ga(x)}&=
\int_\om 
\frac{a(x)}{1-\ga(x)}(z_0)_+^{1-\ga(x)}
\end{aligned}
  \end{equation}
and 
\begin{equation}\label{eq21}
\begin{aligned}
    \lim_{n\to\infty}
\int_\om 
\frac{b(x)}{r(x)}(z_n)_+^{r(x)}&=
\int_\om 
\frac{b(x)}{r(x)}(z_0)_+^{r(x)};\\
    \lim_{n\to\infty}
\int_\om b(x)(z_n)_+^{r(x)}&=
\int_\om b(x)(z_0)_+^{r(x)}.
\end{aligned}
  \end{equation}
  First, we prove that for $z_0\in \Nh^-$
it holds that $R(u_0)>0.$
Since $z_0\in \Nh^-,$
we have
   \begin{align*}
0>\varphi''_{z_0}(1)&\ge p^-P(z_0)
-\la(1-\ga^-)Q(z_0)-r^+R(z_0)\\
&\ge (p^--(1-\ga^-))P(z_0)
-(r^+-(1-\ga^-))R(z_0),
    \end{align*}
    that is,
$$
    R(z_0)\ge 
\frac{(p^--(1-\ga^-))}{(r^+-(1-\ga^-))}P(z_0)>0.
$$
   Next, we claim that $z_0\neq0.$
On the contrary, let us assume that $z_0=0.$
Then from \eqref{eq21}, we get
    \begin{equation*}
\lmt R(z_n)=\lim_{n\to\infty}
\int_\om b(x)(z_n)_+^{r(x)}=
\int_\om b(x)(z_0)_+^{r(x)}=0.
  \end{equation*}
Since $z_n\in \Nh^-,$
hence from Lemma \ref{lm:Nehari_negative}, 
we get
\begin{align*}
        0<K\le J_\la(z_n)\le \Big (
\frac{1}{p^-}-
\frac{1}{1-\ga^-} \Big ) P(z_n)+\Big (
\frac{1}{1-\ga^-}-
\frac{1}{r^+} \Big ) R(z_n)\le0,
    \end{align*}
    which is a contradiction.
Hence, we conclude that 
$z_0\in\X\setminus\{0\}.$

     It only remains to prove that 
     $z_n\ra z_0$
strongly in $\X.$
Suppose not.
Then $z_n\nrightarrow z_0$
in $\X$
as $n\to\infty.$
Therefore, by applying 
Lemma \ref{lem:modular_cgs_prob} and
\cite[Lemma 2.4.17]{Variable_book}, 
up to a subsequence it follows that
        \begin{align}\label{eq23}
        &
\int_\om 
\frac{1}{p(x)}|\grad z_0|^{p(x)}
 +\iint_{\mathbb{R}^{2N}} 
\frac{1}{p(x,y)}
\frac{|z_0(x)-z_0(y)|^{p(x,y)}}
{|x-y|^{N+sp(x,y)}} \\ \nonumber 
        & <\liminf_{n\to\infty}
\int_\om 
\frac{1}{p(x)}|\grad z_n|^{p(x)}
+\iint_{\mathbb{R}^{2N}} 
\frac{1}{p(x,y)}
\frac{|z_n(x)-z_n(y)|^{p(x,y)}}
{|x-y|^{N+sp(x,y)}}.
    \end{align}
    Using \eqref{eq:eneryg_functional},
\eqref{eq29}, \eqref{eq21}, and \eqref{eq23},
 we get
    \begin{align}\label{eq24}
\lim_{n\to \infty} 
& J_\la(z_n)
=\liminf_{n\to \infty}
\Big [
\int_\om 
\frac{1}{p(x)}|\grad z_n|^{p(x)}
+\iint_{\mathbb{R}^{2N}} 
\frac{1}{p(x,y)}
\frac{|z_n(x)-z_n(y)|^{p(x,y)}}
{|x-y|^{N+sp(x,y)}}  \nonumber\\
        & -\la
\int_\om 
\frac{a(x)}{1-\ga(x)}(z_n)_+^{1-\ga(x)}
        -
\int_\om 
\frac{b(x)}{r(x)}(z_n)_+^{r(x)}\Big]
\nonumber\\
    &    \geq \liminf_{n\to \infty}
 \Big [
\int_\om 
\frac{1}{p(x)}|\grad z_n|^{p(x)}
+\iint_{\mathbb{R}^{2N}} 
\frac{1}{p(x,y)}
\frac{|z_n(x)-z_n(y)|^{p(x,y)}}
{|x-y|^{N+sp(x,y)}}\Big ]\nonumber\\
 & -\liminf_{n\to\infty}\Big [\la
\int_\om 
\frac{a(x)}{1-\ga(x)}(z_n)_+^{1-\ga(x)}
\Big ]-\liminf_{n\to\infty}\Big [
\int_\om 
\frac{b(x)}{r(x)}(z_n)_+^{r(x)}
\Big ]\nonumber\\
 &       >
\int_\om 
\frac{1}{p(x)} |\grad z_0|^{p(x)}
+\iint_{\mathbb{R}^{2N}}
\frac{1}{p(x,y)}
\frac{|z_0(x)-z_0(y)|^{p(x,y)}}
{|x-y|^{N+sp(x,y)}}\nonumber\\
                &-\la
\int_\om 
\frac{a(x)}{1-\ga(x)}(z_0)_+^{1-\ga(x)}    
\int_\om 
\frac{b(x)}{r(x)}(z_0)_+^{r(x)}\nonumber\\
   &     =J_\la(z_0).
    \end{align}
    By applying Lemma 
    \ref{lm:fiber_map_analysis} for 
    $u_0\in \X\setminus\{0\}$
there exists $0<t_0^-\in \real^+$
such that $t_0^-u_0\in\Nh^-.$
Since we have assumed that 
$z_n\nrightarrow z_0$
in $\X$
as $n\to\infty,$
it follows that
\begin{equation}\label{eq28}
    \rh(t_0^-z_0)< \liminf_{n\to\infty} 
    \rh(t_0^-z_n).
\end{equation}
Also, using Lemma \ref{lem:modular_cgs_lp} 
and Lebesgue dominated convergence theorem, 
we get
\begin{equation}\label{eq25}
    \lim_{n\to\infty}
\int_\om a(x)(t_0^-z_n)_-^{1-\ga(x)}=
\int_\om a(x)(t_0^-z_0)_+^{1-\ga(x)}
  \end{equation}
and 
\begin{equation}\label{eq26}
    \lim_{n\to\infty}
\int_\om b(x)(t_0^-z_n)_+^{r(x)}=
\int_\om b(x)(t_0^-z_0)_+^{r(x)}.
  \end{equation}
From equations 
\eqref{eq:fibering_map_first_derivative}, 
\eqref{eq28}, \eqref{eq25}, and \eqref{eq26}, 
we conclude that
    \begin{align}\label{eq27}
        \lmt 
        & \varphi'_{z_n}(t_0^-)
  =\liminf_{t\to\infty}\Big [
\int_\om (t_0^-)^{p(x)-1}|\grad (z_n)|^{p(x)}
 \notag \\
 & \quad 
 +\iint_{\mathbb{R}^{2N}} (t_0^-)^{p(x,y)-1}
\frac{|(z_n)(x)-(z_n)(y)|^{p(x,y)}}
{|x-y|^{N+sp(x,y)}}  \nonumber\\
        &\quad -\la
\int_\om (t_0^-)^{-\ga(x)}a(x)(z_n)_+^{1-\ga(x)}
        -
\int_\om (t_0^-)^{r(x)-1}b(x)(z_n)_+^{r(x)}
\Big ]\nonumber\\
       & \ge
\frac{1}{(t_0^-)^+}
\Big  [\liminf_{n\to\infty}\rh((t_0^-)^+z_n)-\la\lmt
\int_\om (t_0^-)^{1-\ga(x)}a(x)(z_n)_+^{1-\ga(x)}
 \nonumber\\
        &\quad -\lmt
\int_\om (t_0^-)^{r(x)}b(x)(z_n)_+^{r(x)}
\Big ]\nonumber\\
        & >
\frac{1}{(t_0^-)^+}\Big [\rh((t_0^-)^+z_0) -\la
\int_\om (t_0^-)^{1-\ga(x)}a(x)(z_0)_+^{1-\ga(x)} \notag \\
 & \quad -
\int_\om (t_0^-)^{r(x)}b(x)(z_0)_+^{r(x)}
\Big ]\nonumber\\
        &=\varphi'_{z_0}(t_0^-)=0.
    \end{align}
    Hence, for sufficiently large $n,$
we have $\ds \lmt \varphi'_{z_n}(t_0^-)>0.$
Since $z_n\in \Nh^-$
for all $n\in\mathbb{N},$
it holds that $\varphi_{z_n}'(1)=0$
and $\varphi_{z_n}''(1)<0.$
By Lemma \ref{lm:fiber_map_analysis}, 
this implies that $\varphi_{u_n}'(t)<0$
for all $t>1.$
Therefore, from \eqref{eq27}, 
we conclude that $t_0^-<1.$
Since $t_0^-z_0\in\Nh^-$
and by Lemma \ref{lm:fiber_map_analysis}, 
$\varphi_{z_n}(t)$
has global maximum at $1.$
Thus, from \eqref{eq24}, 
it follows that
$$
J_\la(t_0^-z_0)<\lmt J_\la(t_0^-z_n)
\le 
\lmt J_\la(z_n)=\inf_{z\in\Nh^-}J_\la(z),
$$
which contradicts the fact that 
$t_0^-z_0\in\Nh^-.$
Hence, $z_n\rightarrow z_0$
strongly in $\X$
as $n\to\infty$
and thus $z_0\in\Nh.$
But by Lemma(refer) $\Nh^0=\{0\}$
and since by Lemma \ref{lm:Nehari_negative} 
$J_\la(z_0)=\ds \lmt J_\la(z_n)>0,$
we conclude that $z_0\in\Nh^-.$
\end{proof}

For the remainder of the paper, 
we will denote by $u\in \Nh^+$
and $v\in\Nh^-$
the minimizers obtained in Lemma 
\ref{lm:min_pos} and \ref{lm:min_neg}, 
respectively.

\begin{lemma}
    For every $0\le\phi\in\X,$
we have $a(x)u_+^{-\ga(x)}\phi\in L^1(\om)$
and 
\begin{align}\label{eq16}
&
\int_\om|\grad u|^{p(x)-2}\na u\na\phi \notag \\
& +
 \iint_{\mathbb{R}^{2N}} 
\frac{|u(x)-u(y)|^{p(x,y)-2}(u(x)-u(y))
(\phi(x)-\phi(y))}{|x-y|^{N+sp(x,y)}}\nonumber\\
    &-\la
\int_\om a(x)u_+^{-\ga(x)}\phi -
\int_\om b(x)u_+^{r(x)-1}\phi\ge0.
    \end{align}
    Moreover, $u>0$
a.e. in $\om$.
\end{lemma}

\begin{proof}[\bf Proof.]
    Let $u\in \Nh^+.$
From Lemma \ref{lm9}, there exists 
$\ve_0>0$
and 
$f_{\ve_0}:B_{\varepsilon_{0}}(0)\ra \R^+$
such that 
$f_{\ve_0}(\ve \phi)(u+\ve\phi)\in\Nh^+$
with 
$$
J_\la(u+\ve\phi)\ge J_\la(f_{\ve_0}(\ve \phi)
(u+\ve\phi))\geq J_\la(u)
$$
for all $0\leq \ve \lv \phi \rv_{\X}<\ve_0.$
Hence,
    $$J_\la(u+\ve\phi)-J_\la(u)\ge0.$$
    From this dividing by $\ve,$
we get
    \begin{align}\label{eq15}
&    \la
\int_\om 
\frac{a(x)}{1-\ga(x)} 
\frac{(u+\ve\phi)_+^{1-\ga(x)}-u_+^{1-\ga(x)}}{\ve} 
\notag \\
&\le    
\int_\om 
\frac{1}{p(x)}
\frac{|\grad (u+\ve\phi)|^{p(x)}-
|\grad u|^{p(x)}}{\ve}\nonumber\\
        &+\iint_{\mathbb{R}^{2N}} 
\frac{1}{p(x,y)}
\frac{|(u+\ve\phi)(x)-(u+\ve\phi)(y)|^{p(x,y)}
-|u(x)-u(y)|^{p(x,y)}}{|x-y|^{N+sp(x,y)}\ve}
\nonumber\\
        &-
\int_\om 
\frac{b(x)}{r(x)}
\frac{(u+\ve\phi)_+^{r(x)}-u_+^{r(x)}}{\ve}
    \end{align}
    Let us consider the first integral 
    on the right-hand side.
For $\ve>0,$
using the mean value theorem there exists 
$0<\theta\in \real$
with $0<\theta<\ve$
such that
 \begin{align}\label{eq14}
 &
\frac{1}{p(x)}
\frac{|\grad (u+\ve\phi)|^{p(x)}-
|\grad u|^{p(x)}}{\ve}
 =|\grad u+
\theta\grad\phi|^{p(x)-2}(\grad u+\theta\grad\phi)
\grad\phi\nonumber\\
    & \le |\grad u+\theta\grad\phi|^{p(x)-1}|\grad\phi|
  \le(|\grad u|+|\grad\phi|)^{p(x)-1}|\grad\phi|.
    \end{align}
    Now, applying H\"older inequality, 
    we get
\begin{align*}
&
\int_\om (|\grad u|+
|\grad\phi|)^{p(x)-1}|\grad\phi| \\
&
  \le 2 \||\grad u|+
    |\grad\phi|)^{p(x)-1}\|_{L^{p'(x)}(\om)}
    \||\grad\phi|\|_{L^{p(x)}(\om)}\\
    &\le 2 \||\grad u|+
    |\grad\phi|)\|^{(p^+-1)}_{L^{p(x)}(\om)}
    \||\grad\phi|\|_{L^{p(x)}(\om)}\\
    &\le 2^{p^+} 
    \Big (\|\grad u\|_{L^{p(x)}(\om)}+
    \|\grad\phi\|_{L^{p(x)}(\om)}\Big )^{(p^+-1)}
    \||\grad\phi|\|_{L^{p(x)}(\om)}.
    \end{align*}
Therefore, 
$(|\grad u|+|\grad\phi|)^{p(x)-1}|
\grad\phi|\in L^1(\om).$
Hence, from \eqref{eq14}, 
we can use Lebesgue 
dominated convergence theorem to get
\begin{align*}
&
    \lim_{\ve\to 0^+}
\int_\om
\frac{1}{p(x)}
\frac{|\grad (u+\ve\phi)|^{p(x)}-
|\grad u|^{p(x)}}{\ve} \\
&=
\int_\om\lim_{\theta\to0^+}|\grad u
+\theta\grad\phi)|^{p(x)-2}(\grad u
+\theta\grad\phi)\grad\phi \nonumber\\
    &=
\int_\om |\grad u|^{p(x)-2}\grad u\grad\phi.
\end{align*}
   Similarly, letting $\ve\to 0^+$
in each of the remaining terms on the 
right-hand side of \eqref{eq15} shows that 
they have finite limit values.
We also observe that, for each $x\in \om$
the difference quotient 
$$
\frac{(u+\ve\phi)_+^{1-\ga(x)}-u_+^{1-\ga(x)}}{\ve}
$$
   increases monotonically as $\ve\to 0^+$
(since $u^{1-\ga(x)}$
is concave for $0<\ga(x)<1$) and 
     \begin{equation*}
        \lim_{\ve\to0^+}
\frac{(u+\ve\phi)_+^{1-\ga(x)}-u_+^{1-\ga(x)}}{\ve}=
        \begin{cases}
0 &\text{ if } \phi=0,\\
u_+^{-\ga(x)}\phi 
&\text{ if }\phi>0 \text{ and }u_+>0,\\
\infty
 &\text{ if }\phi>0 \text{ and }u_+=0.
        \end{cases}
    \end{equation*}
    Hence, by the monotone 
convergence theorem, we obtain 
$a(x)u_+^{-\ga(x)}\phi\in L^1(\om),u>0
\text{ a.e. in }\om$
and 
\begin{align*}
\int_\om|\grad u|^{p(x)-2}\na u\na\phi
+&\iint_{\mathbb{R}^{2N}}
\frac{|u(x)-u(y)|^{p(x,y)-2}(u(x)
-u(y))(\phi(x)-\phi(y))}{|x-y|^{N+sp(x,y)}}\\
    &-\la
\int_\om a(x)u_+^{-\ga(x)}\phi -
\int_\om b(x)u_+^{r(x)-1}\phi\ge0.
    \end{align*}
    Using Lemma \ref{lm9} and 
\ref{lm:min_neg} and reasoning as above, 
we conclude that \eqref{eq16} also holds 
for $v\in\Nh^-$.
\end{proof}

\begin{lemma}\label{lm:existence_result}
    The functions $u$
and $v$
are positive weak solutions of \eqref{eq1}.
\end{lemma}

\begin{proof}[\bf Proof.]
    For any $\phi\in \X$
and $\ve>0,$
set $\Phi(x)=(u+\ve\phi)_+(x).$
Let $\om=\om_1\times\om_2$
where
$$
    \om_1:=\{x\in\om:u+\ve\phi>0\}
    \text{ and } 
    \om_2:=\{x\in\om:u+\ve\phi\le0\}.
    $$
    We can see that $\Phi|_{\om_1}=u+\ve\phi$
and $\Phi|_{\om_2}=0.$
Next, we decompose $\Q$
in the following way
\begin{align*}
\Q:= (\om_1\times\om^c) 
& \cup(\om_2\times\om^c)
\cup(\om^c\times\om_1)  
  \cup (\om^c\times\om_2) \\ &
\cup(\om_2\times\om_1)\cup
 (\om_1\times\om_2)   \cup(\om_1\times\om_1)\cup
  (\om_2\times\om_2).
\end{align*}
For convenience, 
we denote 
$$
d \nu := \frac{dx \, dy}{|x-y|^{N+sp(x,y)}},
$$
and
$$
W(x,y):=U(x,y)[(u+\ve\phi)_-(x)-(u+\ve\phi)_-(y)] ,
$$
where $U(x,y):=|u(x)-u(y)|^{p(x,y)-2}(u(x)-u(y))$.

Thus, we have
\begin{enumerate}
\item[(i)]$\ds
\int_{\om_1\times\om^c}W(x,y)\,d\nu=
\int_{\om^c\times\om_1}W(x,y)\,d\nu=0.$

\item[(ii)]$\ds 
\int_{\om_2\times\om^c}W(x,y)\,d\nu=-
\int_{\om_2\times\om^c}|u(x)|^{p(x,y)-2}u(x)
(u+\ve\phi)(x) \,d\nu.$

\item[(iii)]$\ds 
\int_{\om^c\times\om_2}W(x,y)\,d\nu=-
\int_{\om_2\times\om^c}|u(x)|^{p(x,y)-2}u(x)
(u+\ve\phi)(x) \,d\nu.$

\item[(iv)]$\ds 
\int_{\om_2\times\om_1}W(x,y)\,d\nu=-
\int_{\om_2\times\om_1}U(x,y)(u+\ve\phi)
(x) \,d\nu.$

\item[(v)]$\ds 
\int_{\om_1\times\om_2}W(x,y)\,d\nu=-
\int_{\om_2\times\om_1}U(x,y)
(u+\ve\phi)(x) \,d\nu.$

\item[(vi)]$\ds 
\int_{\om_1\times\om_1}W(x,y)\,d\nu=0.$

\item[(vii)]$\ds 
\int_{\om_2\times\om_2}W(x,y)\,d\nu=-
\int_{\om_2\times\om_2}U(x,y)((u+\ve\phi)(x)
-(u+\ve\phi)(y)) \,d\nu.$
\end{enumerate}
Substituting $\Phi$
in \eqref{eq16}, we get
\begin{align*}
    0\le&
\int_\om|\grad u|^{p(x)-2}\na u\na\Phi+
\int_{\Q} 
\frac{U(x,y)(\Phi(x)
-\Phi(y))}{|x-y|^{N+sp(x,y)}}\nonumber -\la
\int_\om a(x)u_+^{-\ga(x)}\Phi \\
&-\int_\om b(x)u_+^{r(x)-1}\Phi\\
    =&
\int_\om|\grad u|^{p(x)-2}\na u\na(u+\ve\phi)+
\int_{\Q} 
\frac{U(x,y)((u+\ve\phi)(x)-
(u+\ve\phi)(y))}{|x-y|^{N+sp(x,y)}}\\
    &-\la
\int_\om a(x)u_+^{-\ga(x)}(u+\ve\phi) -
\int_\om b(x)u_+^{r(x)-1}(u+\ve\phi) \\
 &+
\int_\om|\grad u|^{p(x)-2}\na u\na(u+\ve\phi)_- 
+
\int_{\Q} 
\frac{W(x,y)}
{|x-y|^{N+sp(x,y)}} \\
&   -\la
\int_\om a(x)u_+^{-\ga(x)}(u+\ve\phi)_- -
\int_\om b(x)u_+^{r(x)-1}(u+\ve\phi)_-\\
    =&\ve \left\{
\int_\om|\grad u|^{p(x)-2}\na u\na\phi+
\int_{\Q} 
\frac{U(x,y)(\phi(x)-\phi(y))}{|x-y|^{N+sp(x,y)}}-\la
\int_\om a(x)u_+^{-\ga(x)}\phi \right. \\
& \left. -
\int_\om b(x)u_+^{r(x)-1}\phi\right\} +
\left\{ \int_\om|\grad u|^{p(x)}+
\int_{\Q} 
\frac{|u(x)-u(y)|^{p(x,y)}}{|x-y|^{N+sp(x,y)}} \right.\\
& \left. -\la
\int_\om a(x)u_+^{1-\ga(x)}-
\int_\om b(x)u_+^{r(x)} \right\} -
\int_{\om_2}|\grad u|^{p(x)}\\
    &-
\int_{\om_2}\ve|\na u|^{p(x)-2}\na u\na \phi+
\int_{\Q} 
\frac{W(x,y)}{|x-y|^{N+sp(x,y)}}\\
    &+\la
\int_{\om_2} a(x)u_+^{-\ga(x)}(u+\ve\phi)+
\int_{\om_2} b(x)u_+^{r(x)-1}(u+\ve\phi)\\
    \le& \ve I_1-
\int_{\om_2}\ve|\na u|^{p(x)-2}\na u\na \phi+
\int_{\Q} 
\frac{U(x,y)((u+\ve\phi)_-(x)-(u+\ve\phi)_-(y))}
{|x-y|^{N+sp(x,y)}}\\
    &+
\int_{\om_2} b(x)u_+^{r(x)-1}(u+\ve\phi),
\end{align*}
where 
\begin{equation*}
\begin{aligned}
I_1 =
\int_\om&|\grad u|^{p(x)-2}\na u\na\phi+
\int_{\Q} 
\frac{U(x,y)(\phi(x)-\phi(y))}{|x-y|^{N+sp(x,y)}} \\
&-\la
\int_\om a(x)u_+^{-\ga(x)}\phi -
\int_\om b(x)u_+^{r(x)-1}\phi.
\end{aligned}
\end{equation*}

\noindent Using above together with the estimates obtained for $W(x,y)$ gives 
    \begin{align*}
&0  \le \ve I_1-
\int_{\om_2}\ve|\na u|^{p(x)-2}\na u\na \phi+
\int_{\om_2} b(x)u_+^{r(x)-1}(u+\ve\phi) \\
&\quad -2
\int_{\om_2\times \om_1} 
\frac{U(x,y)(u+\ve\phi)(x)}{|x-y|^{N+sp(x,y)}} -2
\int_{\om_2\times \om^c} 
\frac{U(x,y)u(x)(u+\ve\phi)(x)}{|x-y|^{N+sp(x,y)}}\\
& \quad  -
\int_{\om_2\times \om_2} 
\frac{U(x,y)((u+\ve\phi)(x)-(u+\ve\phi)(y))}
{|x-y|^{N+sp(x,y)}}\\
    & \le  \ve I_1-
\int_{\om_2}\ve|\na u|^{p(x)-2}\na u\na \phi+
\int_{\om_2} b(x)u_+^{r(x)-1}(u+\ve\phi) \\
& \quad -2 \int_{\om_2\times \om_1} 
\frac{U(x,y)u(x)}{|x-y|^{N+sp(x,y)}} -2
\int_{\om_2\times \om^c} 
\frac{|u(x)|^{p(x,y)}}{|x-y|^{N+sp(x,y)}} \\
& \quad  - \int_{\om_2\times \om_2} 
\frac{|u(x)-u(y)|^{p(x,y)}}{|x-y|^{N+sp(x,y)}}-\ve
\left (2
\int_{\om_2\times \om_1} 
\frac{U(x,y)\phi(x)}{|x-y|^{N+sp(x,y)}}\right.\\
    & \quad \left. 2
\int_{\om_2\times \om^c} 
\frac{|u(x)|^{p(x,y)-2}u(x)\phi(x)}{|x-y|^{N+sp(x,y)}}+
\int_{\om_2\times \om_2} 
\frac{U(x,y)(\phi(x)-\phi(y))}{|x-y|^{N+sp(x,y)}}
\right )\\
    & \le  \ve I_1-\ve
\int_{\om_2}|\na u|^{p(x)-2}\na u\na \phi+
\int_{\om_2} b(x)u_+^{r(x)-1}(u+\ve\phi)\\
& \quad  -2
\int_{\om_2\times \om_1} 
\frac{U(x,y)u(x)}{|x-y|^{N+sp(x,y)}}
-2\ve
\left (
\int_{\om_2\times \om_1} 
\frac{U(x,y)\phi(x)}{|x-y|^{N+sp(x,y)}} \right. \\
& \quad \left. +
\int_{\om_2\times \om^c} 
\frac{|u(x)|^{p(x,y)-2}u(x)\phi(x)}
{|x-y|^{N+sp(x,y)}} + \frac{1}{2}
\int_{\om_2\times \om_2} 
\frac{U(x,y)(\phi(x)-\phi(y))}
{|x-y|^{N+sp(x,y)}}
\right )\\
    & \le   \ve I_1-\ve
\int_{\om_2}|\na u|^{p(x)-2}\na u\na \phi+\ve
\int_{\om_2} b(x)u_+^{r(x)-1}\phi \\
& +\ve\|b\|_{L^{
\frac{p^*(x)}{p^*(x)-r(x)}}(\om)}\ve^{r^--1}\Big (
\int_{\om_2}|\phi|^{p^*(x)}\Big )^{
\frac{r(x)}{p^*(x)}}\\
    &+2\ve\Big (
\int_{\om_2\times \om_1} 
\frac{|u(x)-u(y)|^{p(x,y)}}{|x-y|^{N+sp(x,y)}}\Big )^{
\frac{p(x,y)-1}{p(x,y)}}\Big (
\int_{\om_2\times \om_1} 
\frac{|\phi(x)|^{p(x,y)}}{|x-y|^{N+sp(x,y)}}\Big )^{
\frac{1}{p(x,y)}}\\
    &-2\ve\left[\Big (
\int_{\om_2\times \om^c} 
\frac{|u(x)|^{p(x,y)}}{|x-y|^{N+sp(x,y)}}\Big )^{
\frac{p(x,y)-1}{p(x,y)}}\Big (
\int_{\om_2\times \om^c} 
\frac{|\phi(x)|^{p(x,y)}}{|x-y|^{N+sp(x,y)}}\Big )^{
\frac{1}{p(x,y)}}\right.\\
    &+\Big (
\int_{\om_2\times \om_1} 
\frac{|u(x)-u(y)|^{p(x,y)}}{|x-y|^{N+sp(x,y)}}\Big )^{
\frac{p(x,y)-1}{p(x,y)}}\Big (
\int_{\om_2\times \om_1} 
\frac{|\phi(x)|^{p(x,y)}}{|x-y|^{N+sp(x,y)}}\Big )^{
\frac{1}{p(x,y)}}\\
    &\left.+
\frac{1}{2}\Big (
\int_{\om_2\times \om_2} 
\frac{|u(x)-u(y)|^{p(x,y)}}{|x-y|^{N+sp(x,y)}}\Big )^{
\frac{p(x,y)-1}{p(x,y)}}\Big (
\int_{\om_2\times \om_2} 
\frac{|\phi(x)-\phi(y)|^{p(x,y)}}
{|x-y|^{N+sp(x,y)}}\Big )^{
\frac{1}{p(x,y)}}\right].
    \end{align*}
    As the measure of the set 
    $\om_2:=\{x\in\om:u+\ve\phi\le0\}$
tends to $0$
as $\ve\to0.$
Therefore dividing by $\ve$
and letting $\ve\to0$
yields
\begin{align}\label{eq17}
\int_\om|\grad u|^{p(x)-2}\na u\na\phi+
&\iint_{\mathbb{R}^{2N}}
\frac{|u(x)-u(y)|^{p(x,y)-2}(u(x)-u(y))
(\phi(x)-\phi(y))}{|x-y|^{N+sp(x,y)}}\nonumber\\
    &-\la
\int_\om a(x)u_+^{-\ga(x)}\phi -
\int_\om b(x)u_+^{r(x)-1}\phi\ge0.
    \end{align}
Replacing $\phi$
by $-\phi$
in \eqref{eq17} gives the reverse inequality.
Hence, $u$
is a positive weak solution of \eqref{eq2} 
and therefore also a positive weak 
solution of \eqref{eq1}.

For $v\in \Nh^-,$
by applying Lemmas~\ref{lm9} and 
\ref{lm:min_neg}, 
together with similar arguments as above, 
we conclude that $v$
is a positive weak solution of \eqref{eq1}.
\end{proof}

\begin{proof}[\bf Proof of the Theorem \ref{main_theorem}]
The result follows directly from 
Lemmas~\ref{lm:min_pos}, \ref{lm:min_neg}, 
and~\ref{lm:existence_result}.
\end{proof}
Next, we establish the following regularity 
result using the De Giorgi iteration technique 
combined with a localization argument.
We adopt the approach developed in
\cite{KyHo} and
\cite{KuHo2} 
for the mixed local and nonlocal setting.

\begin{proof}[\bf Proof of the Theorem \ref{reg_result}]
    Let $u \in \X(\om)$
be a positive weak solution of \eqref{eq1} 
and define the level set 
$$
\mathcal{A}_k:=\{x\in\om:u>k\},\quad \text{for any } k\ge 1.
$$
 We choose $\Psi_n=(u-k_{n+1})_+\in\X(\om)$
as a test function in \eqref{weak_sol} 
where $$k_n:=k_*
\left (2-
\frac{1}{2^n}
\right )\quad n=0,1,2\dots,$$
with $k_*\ge1.$
Consequently, we obtain
\begin{align} \label{eq18}
&
\int_{\Omega} |\nabla u|^{p(x)-2} 
\nabla u \cdot \nabla \Psi_n \, dx  \notag \\
&+ \iint_{\mathbb{R}^{2N}} 
\frac{|u(x) - u(y)|^{p(x,y)-2} 
(u(x) - u(y)) (\Psi_n(x) - \Psi_n(y))}
{|x - y|^{N + sp(x,y)}} \, dx \, dy\nonumber\\
&=\la
\int_{\Omega} a(x) u^{-\gamma(x)} \Psi_n \, dx+
\int_{\Omega} b(x) u^{r(x)-1}  \Psi_n \, dx.
\end{align}
Let us consider the nonlocal integral
$$
\iint_{\mathbb{R}^{2N}} 
\frac{|u(x) -
 u(y)|^{p(x,y)-2} (u(x) - u(y))
 (\Psi_n(x) - 
 \Psi_n(y))}{|x - y|^{N + sp(x,y)}} \, dx \, dy.
 $$
Using the inequality: For every $f,g,\alpha\in \R$
and $\al>1$
there holds
$$
|f-g|^{\al-2}(f-g)(f_+-g_+)\ge |f_+-g_+|^\al
$$
where $f_+=\max\{f,0\}$
and $g_+=\max\{g,0\}$,
we get
\begin{align}\label{eq19}
|u(x) - u(y)|^{p(x,y)-2} 
(u(x) - u(y)) (\Psi_n(x) - \Psi_n(y))\ge 0.
\end{align}
It follows from \eqref{eq18} and \eqref{eq19} that
\begin{align}\label{eqr16}
\int_{\Omega} |\nabla u|^{p(x)-2} \nabla u 
\cdot \nabla \Psi_n \, dx 
\le \la
\int_{\Omega} a(x) u^{-\gamma(x)} \Psi_n \, dx+
\int_{\Omega} b(x) u^{r(x)-1}  \Psi_n \, dx.
\end{align}
We consider the two cases separately.\\
\textbf{Case(a):} Let $a(x) \in L^\infty(\Omega).$
By H\"older inequality and \eqref{eqr16}, we obtain
\begin{align}
&
\int_{\mathcal{A}_{k_{n+1}}} |\nabla u|^{p(x)} \, dx \notag \\
&
\le  \la\|a\|_{L^\infty(\A)}
\int_{\Omega} u^{-\gamma(x)} \Psi_n \, dx
+\|b\|_{L^\infty(\A)}
\int_{\A} u^{r(x)} \, dx\nonumber\\
\le&C_1
\int_{\A} u^{r(x)} \, dx.
\end{align}
Now, the arguments of
\cite[Theorem 4.2]{KuHo2} apply directly.\\
\textbf{Case(b):} Let $a(x) \in L^{
\frac{r(x)}{r(x) - (1 - \ga(x))}}(\Omega).$
Using H\"older inequality and Lemma \ref{lemA1} in
 \eqref{eqr16}, we get
\begin{align}\label{eqr1}
\int_{\mathcal{A}_{k_{n+1}}}
&  |\nabla u|^{p(x)} \, dx 
\le \la\|a\|_{L^{
\frac{r(x)}{r(x)-(1-\ga(x))}}(\A)}
\|u^{1-\gamma(x)}\|_{L^{
\frac{r(x)}{1-\ga(x)}}(\A)}\nonumber\\
&+\|b\|_{L^\infty(\A)}
\int_{\A} u^{r(x)} \, dx\nonumber\\
\le&C_1
\left (\|u\|_{L^{r(x)}(\A)}^{r^+}+
\|u\|_{L^{r(x)}(\A)}^{r^-}
\right )+C_2
\int_{\A} u^{r(x)} \, dx
\end{align}
Let $\rho(u):L^{r(x)}(\om)\to \R$
denote the modular function.
Then, by Lemma \ref{modular_ineq}, we can write
\begin{align*}
\int_{\mathcal{A}_{k_{n+1}}} |\nabla u|^{p(x)} \, dx 
\le \begin{cases}
    C_1\Big (\ds
\int_{\A} u^{r(x)} \, dx\Big )^{
\frac{r^-}{r^+}} &\text{ if } \rho(u)<1;\\
   C_1\Big (\ds
\int_{\A} u^{r(x)} \, dx\Big )^{
\frac{r^+}{r^-}} &\text{ if } \rho(u)\ge1.\\
\end{cases}
\end{align*}
Next, we define
\[
Z_n:=
\int_{\mathcal{A}_{k_n}} (u-k_n)^{r(x)} \, dx.
\]
Since, $k_*\le k_n \le k_{n+1}<2k^*
\text{ for all } n\in\N,$
using definition of $k_n,$
we obtain
\begin{align*}
    Z_n:=
\int_{\mathcal{A}_{k_n}} (u-k_n)^{r(x)} \, dx 
& \ge 
\int_{\A} u^{r(x)}\Big ( 1-
\frac{k_n}{k_{n+1}}\Big )^{r(x)} \, dx \\
&
\ge 
\int_{\A}
\frac{u^{r(x)}}{2^{r(x)(n+2)}} \, dx.
\end{align*}
Therefore,
\begin{equation}\label{eqr2}
\int_{\A}u^{r(x)} \, dx\le d_1\al^nZ_n,
\end{equation}
where $d_1=2^{2r^+}, \al=2^{r^+}>1.$
From \eqref{eqr1} and \eqref{eqr2}, we have
\begin{align}\label{eqr3}
\int_{\mathcal{A}_{k_{n+1}}} 
|\nabla (u-k_{n+1})|^{p(x)} \, dx 
\le \begin{cases}
    C_1\Big ( d_1\al^nZ_n\Big )^{
\frac{r^-}{r^+}} &\text{ if } \rho(u)<1;\\
   C_1\Big ( d_1\al^nZ_n\Big )^{
\frac{r^+}{r^-}} &\text{ if } \rho(u)\ge1.\\
\end{cases}
\end{align}
Next, we estimate the Lebesgue measure of $\A$
in the following way
\begin{align*}
    |\A|\le  
\int_{\A} \Big ( 
\frac{u-k_n}{k_{n+1}-k_n}\Big )^{r(x)} \, dx\le 
\int_{\mathcal{A}_{k_n}}
\frac{2^{r(x)(n+1)}}{k_*^{r(x)}} (u-k_n)^{r(x)}\, dx.
\end{align*}
This implies that 
\begin{equation}\label{eqr4}
    |\A|\le  
\frac{\al}{k_*^{r^-}}\al^nZ_n.
\end{equation}
We follow the localization arguments of
\cite[Theorem 4.2, page no.\ 159]{KuHo2}.
For any $R>0,$
let $\{B_i(R)\}_{i=1}^m$
be a finite open cover of $\oline\om$
such that $\oline\om\subset\cup_{i=1}^mB_i,$
here $B_i:=B_i(R)$
represents a ball with radius $R.$
We denote $p_i^+=\max_{B_i\cap\oline\om }p(x),$
$p_i^-=\min_{B_i\cap\oline\om }p(x),$
$r_i^+=\max_{B_i\cap\oline\om }r(x),$
$r_i^-=\min_{B_i\cap\oline\om }r(x)$
and $(p_i^-)^*=
\frac{Np_i^-}{N-p_i^-}.$
Since $p(x)\le r(x)<p^*(x)$
for all $x\in\oline\om$
and $p,r$
are continuous in $\oline\om,$
we can choose $R>0$
small enough such that $p_i^+\le r_i^+<(p_i^-)^*$.

Let $\{\xi_i\}_{i=1}^m$
be a partition of unity of $\oline\om$
associated with the open cover $\{B_i(R)\}_{i=1}^m,$
that is, for each $i=1,\dots,m,$
$\xi_i\in C_0^\infty(\R^N),$
$\text{supp}(\xi_i)\subset B_i, 0\le \xi_i\le 1$
and $ \sum_{i=1}^m\xi_i=1.$
Let $l>0$
be a constant such that $|\na\xi_i|\le l$
for $i=1,\dots,m.$
To proceed, we establish the following 
estimate for the gradient term
\begin{align}
\label{eqr5}
&
   \sum_{i=1}^m 
\int_{\mathcal{A}_{k_{n+1}}} 
|\xi_i\nabla (u-k_{n+1})|^{p_i^-} \, dx \\
\notag
& \le |\A|+
\int_{\mathcal{A}_{k_{n+1}}} 
|\nabla (u-k_{n+1})|^{p(x)} \, dx . 
\end{align}
Using \eqref{eqr3} and \eqref{eqr4} 
in \eqref{eqr5}, we get
\begin{align}
\label{eqr6}
&
   \sum_{i=1}^m 
\int_{\mathcal{A}_{k_{n+1}}} 
|\xi_i\nabla (u-k_{n+1})|^{p_i^-}
 \, dx \\ \notag
 &
 \le \begin{cases}
\frac{\al}{k_*^{r^-}}\al^nZ_n+C_1
\Big ( d_1\al^nZ_n\Big )^{
\frac{r^-}{r^+}} &\text{ if } \rho(u)<1;\\
   \frac{\al}{k_*^{r^-}}\al^nZ_n+C_1
   \Big ( d_1\al^nZ_n\Big )^{
\frac{r^+}{r^-}} &\text{ if }
 \rho(u)\ge1. 
\end{cases} 
\end{align}
Next, we estimate $Z_{n+1}$
by $Z_n.$
Using the partition of unity and Jensen's
 inequality, we have
\begin{equation*}
    \begin{aligned}
Z_{n+1} & =
\int_{\A}
  (u-k_{n+1}
  )^{r(x)} d x=
\int_{\A}
  (u-k_{n+1}
  )^{r(x)}
\Big (\sum_{i=1}^{m} \xi_{i}
\Big )^{r^{+}} dx \\
& \leq m^{r^{+}-1} \sum_{i=1}^{m} 
\int_{\A}
\left (u-k_{n+1}
\right )^{r(x)} \xi_{i}^{r^{+}} dx \\
& \leq m^{r^{+}-1} \sum_{i=1}^{m} 
\int_{\A}
\left (u-k_{n+1}
\right )^{r(x)} \xi_{i}^{r(x)} dx.
\end{aligned}
\end{equation*}
Thus,
\begin{equation}\label{eqr7}
Z_{n+1}\leq m^{r^{+}-1} \sum_{i=1}^{m} 
\Big [
\int_{\A}
  (u-k_{n+1}
  )^{r_i^+} \xi_{i}^{r_i^+} dx+
\int_{\A}
  (u-k_{n+1}
  )^{r_i^-} \xi_{i}^{r_i^-} dx
  \Big].
\end{equation}
Let $i\in\{1,2,3,\dots,m\}$
be fixed and suppose that $q\in\{r_i^+,r_i^-\},$
that is, $p_i^-\le q\le (p_i^-)^*.$
By H\"older and Sobolev inequalities, we derive
\begin{align}\label{eqr10}
\int_{\A} &
\left (u-k_{n+1}
\right )^{q} \xi_{i}^{q} dx \\
&\le \Big (
\int_{\om_i}\lsb\xi_i
\left (u-k_{n+1}
\right )_+\rsb^{(p_i^-)^*} dx\Big )^{
\frac{q}{(p_i^-)^*}}|\A|^{1-
\frac{q}{(p_i^-)^*}}\nonumber\\
    &\le C_i^q\Big (
\int_{\om_i}\lsb\na(\xi_i
\left (u-k_{n+1}
\right )_+)\rsb^{p_i^-} dx\Big )^{
\frac{q}{p_i^-}}|\A|^{1-
\frac{q}{(p_i^-)^*}}\nonumber\\
    &\le C_i^q
\left ( 2^{p_i^--1}
\int_{\A}|\xi_i\na
\left (u-k_{n+1}
\right )_+|^{p_i^-}dx\right.\nonumber\\
 &\quad\left.+2^{p_i^--1}
\int_{\A}|\na\xi_i|^{p_i^-}
\left (u-k_{n+1}
\right )_+^{p_i^-}dx
\right )^{
\frac{q}{p_i^-}}|\A|^{1-
\frac{q}{(p_i^-)^*}}
\end{align}
here $C_i$
is the embedding constant.
We also have 
\begin{equation}\label{eqr8}
\int_{\om_i}|\xi_i
\left (u-k_{n+1}
\right )_+|^{p_i^-} dx\le 
\int_{\A}u^{r(x)} dx\le d_1\al^nZ_n.
\end{equation}
From \eqref{eqr6}, \eqref{eqr10} and \eqref{eqr8},
we get the estimate
\begin{flalign}\label{eqr9}
   &
\int_{\A}
\left (u-k_{n+1}
\right )^{q} \xi_{i}^{q} dx\nonumber\\
    &\le\begin{cases}
    |\A|^{1-
\frac{q}{(p_i^-)^*}}C_i^q  \left[ 2^{p_i^--1}\Big ( 
\frac{\al}{k_*^{r^-}}\al^nZ_n \right. \\
\qquad \qquad \qquad \qquad  \left. + 
C_1 \Big ( d_1\al^nZ_n\Big )^{
\frac{r^-}{r^+}} +d_1l^{p_i^-}\al^nZ_n\Big )\right]^{
\frac{q}{p_i^-}} 
\text{ if }\rho(u)<1;\\[10pt]
  |\A|^{1-
\frac{q}{(p_i^-)^*}}C_i^q\left[ 2^{p_i^--1}\Big ( 
\frac{\al}{k_*^{r^-}}\al^nZ_n \right. \\
\qquad \qquad \qquad \qquad
\left. +C_1\Big ( d_1\al^nZ_n\Big )^{
\frac{r^+}{r^-}}+d_1l^{p_i^-}\al^nZ_n\Big )\right]^{
\frac{q}{p_i^-}} 
\text{ if } \rho(u)\ge1.
\end{cases} 
\end{flalign}
Since $Z_{n}^{\frac{q}{p_{i}^{-}}} 
\leq Z_{n}+Z_{n}^{\frac{r^{+}}{p^{-}}},$
it follows from \eqref{eqr4} that
\begin{equation}\label{eqr11}
\begin{aligned}
\left |\A\right|^{1-
\frac{q}{
  (p_{i}^{-}
  )^{*}}} &\leq 
\frac{\al^{1-\frac{q}{
  (p_{i}^{-}
  )^{*}}}}{
  (k_{*}^{r^{-}}
  )^{1-\frac{q}{
  (p_{i}^{-}
  )^{*}}}}
\Big (\al^{1-\frac{q}{
  (p_{i}^{-}
  )^{*}}}
\Big )^{n} Z_{n}^{1-
\frac{q}{
  (p_{i}^{-}
  )^{*}}} \\
&\leq \frac{\al^{1-
\frac{q}{
  (p_{i}^{-}
  )^{*}}}}{k_{*}^{r^{-}(1-\eta)}}
\Big (\al^{1-\frac{q}{
  (p_{i}^{-}
  )^{*}}}
\Big )^{n}
  (Z_{n}+Z_{n}^{1-\eta}
  ),
\end{aligned}
\end{equation}
where, $\eta:=\ds\max_{1\le i\le m} 
\frac{r_i^+}{(p_i^-)^*}<1.$
To estimate $Z_{n+1},$
we first consider the case $\rho(u)<1.$
From \eqref{eqr7}, \eqref{eqr9}, and \eqref{eqr11},
 we deduce
\begin{align}\label{eqr12}
    Z_{n+1}&\le 
\frac{C_0}{k_{*}^{r^{-}(1-\eta)}}
\Big (\al^{1-
\frac{q}{
  (p_{i}^{-}
  )^{*}}}
\Big )^{n}\al^{
\frac{nr^+}{p^-}}
\Big (Z_{n}^{1+
\frac{r^+}{p^-}}+Z_{n}^{1-\eta+
\frac{(r^-)^2}{r^+p^-}}
\Big )\nonumber\\
    &\le 
\frac{C_0}{k_{*}^{r^{-}(1-\eta)}}\beta^{n}
\Big (Z_{n}^{1+
\frac{r^+}{p^-}}+Z_{n}^{1-\eta+
\frac{(r^-)^2}{r^+p^-}}
\Big ),
\end{align}
here, $C_0$
is a constant independent of $n,k_*$
and $\beta=\al^{1+
\frac{r^+}{p^-}}>1.$
Alternatively, \eqref{eqr12} takes the form
\begin{align}\label{eqr13}
    Z_{n+1}\le 
\frac{C_0}{k_{*}^{r^{-}(1-\eta)}}\beta^{n}
\left (Z_{n}^{1+\delta_2}+Z_{n}^{1+\delta_1}
\right ),
\end{align}
where $0<\delta_1=
\frac{(r^-)^2}{r^+p^-}-\eta<1\le\delta_2=
\frac{r^+}{p^-}.$
Using
\cite[Lemma 4.3]{KuHo3}, along with 
assumption $(i)$
of Theorem \ref{reg_result} and equation 
\eqref{eqr13}, we arrive at
\begin{equation}\label{eqr14}
Z_{n} \leq
\Big (
\frac{2 C_0}{k_{*}^{r^{-}(1-\eta)}}
\Big )^{
\frac{-1}{\delta_{1}}} \beta^{
\frac{-1}{\delta_{1}^{2}}} \beta^{
\frac{-n}{\delta_{1}}}, \quad n \in \mathbb{N}
\end{equation}
if
\begin{equation}\label{eqr15}
Z_{0}=
\int_{\mathcal{A}_{k_*}}
  (u-k_{*}
  )^{r(x)} d x \leq
\Big (
\frac{2 C_0}{k_{*}^{r^{-}(1-\eta)}}
\Big )^{
\frac{-1}{\delta_{1}}} \beta^{
\frac{-1}{\delta_{1}^{2}}}.
\end{equation}
Noticing that 
$$
\ds
\int_{\mathcal{A}_{k_{*}}}
\left (u-k_{*}
\right )^{r(x)} dx<\ds
\int_{\om}u_+^{r(x)} dx.
$$
Therefore, if we choose
$$
k_{*}=\max 
\Big (1,
\Big (
  (2 C_0 )^{
\frac{1}{\delta_{1}}} \beta^{
\frac{1}{\delta_{1}^{2}}} 
\int_{\Omega} u_{+}^{r(x)} d x
\Big )^{
\frac{\delta_{1}}{r^{-}(1-\eta)}}
\Big  )
$$
then \eqref{eqr15} holds.
Taking the limit as $n\to\infty$
in \eqref{eqr14} and applying Lebesgue 
dominated convergence theorem, we deduce
\begin{equation*}
\int_{\om}
\left (u-2k_{*}
\right )_+^{r(x)} dx=0,
\end{equation*}
which implies $
\left (u-2k_{*}
\right )_+=0$
a.e. in $\om.$
In other words, 
$\ds\operatorname*{ess\,sup}_{\Omega} u \leq 2k_*.$
This completes the proof for 
$\rho(u)<1.$
The case $\rho(u) \geq 1$
follows analogously under 
assumption $(ii)$
of Theorem~\ref{reg_result}.
\end{proof}

\noindent
{\bf Acknowledgement.}
The author, Shammi Malhotra, is supported by the 
Prime Minister’s Research Fellowship (PMRF ID - 1403226).

\end{document}